\documentclass{LMCS}

\def\dOi{10(3:9)2014}
\lmcsheading%
{\dOi}
{1--28}
{}
{}
{May\phantom.~30, 2012}
{Aug.~21, 2014}
{}

\ACMCCS{[{\bf Mathematics of computating}]: Discrete mathematics}

\usepackage[british]{babel}
\usepackage[utf8]{inputenc}

\usepackage{amsmath,hyperref}
\numberwithin{equation}{section}
\allowdisplaybreaks

\def\paragraph{\subsection*}

\usepackage{amssymb}
\def\restriction{\mathbin{\upharpoonright}}
\DeclareSymbolFont{symbolsC}{U}{pxsyc}{m}{n}
\DeclareMathSymbol{\diamdot}{\mathord}{symbolsC}{144}

\usepackage{amsthm}

\theoremstyle{remark}
\newtheorem*{note}{Note}

\theoremstyle{definition}

\def\F{\mathcal{F}}
\def\C{\mathcal{C}}

\def\P{\mathcal{P}}

\def\I{\mathfrak{I}}
\def\Hh{\mathfrak{H}}
\def\M{\mathfrak{M}}
\def\N{\mathfrak{N}}

\let\phi\varphi

\def\eto{\stackrel{\lower.3ex\hbox{\!$\scriptstyle\mathrm{e}$}}{\rightarrow}}
\def\hto{\stackrel{\lower.4ex\hbox{\!$\scriptstyle\mathrm{h}$}}{\rightarrow}}
\let\pieq\approx
\def\from{\ \leftarrow\ }

\DeclareMathOperator{\Csp}{CSP}

\DeclareMathOperator{\dom}{dom}
\DeclareMathOperator{\Inc}{Inc}
\DeclareMathOperator{\Gai}{Gai}
\DeclareMathOperator{\ar}{ar}
\DeclareMathOperator{\UP}{UP}
\DeclareMathOperator{\Forbh}{Forb_{h}}

\def\ocat{$\omega$-cat\-e\-gor\-i\-cal}
\def\sutst{$(\sigma\cup\tau)$-struc\-ture}
\def\sutordst{$(\sigma\cup\tau\cup\{{\preceq}\})$-struc\-ture}
\def\sigord{$(\sigma\cup\{{\preceq}\})$}
\def\sigst{$\sigma$-struc\-ture}
\def\taust{$\tau$-struc\-ture}

\def\ol#1{\bar #1}
\def\id{{\mathrm{id}}}
\def\goal{\ifmmode\mathbf{goal}\else$\mathbf{goal}$\fi}
\def\Goal{\mathbf{goal}}
\def\true{\ifmmode\mathbf{true}\else$\mathbf{true}$\fi}
\def\false{\ifmmode\mathbf{false}\else$\mathbf{false}$\fi}
\providecommand{\abs}[1]{\lvert#1\rvert}

\begin{document}

\title{On Ramsey~properties of classes with~forbidden~trees}

\author[J.~Foniok]{Jan Foniok}%
\address{Manchester Metropolitan University, School of Computing, Mathematics and Digital Technology, Chester Street, Manchester M1 5GD, England}
\email{j.foniok@mmu.ac.uk}


\begin{abstract}
Let $\F$ be a set of relational trees and let $\Forbh(\F)$ be the
class of all structures that admit no homomorphism from any tree
in~$\F$; all this happens over a fixed finite relational
signature~$\sigma$. There is a natural way to expand $\Forbh(\F)$
by unary relations to an amalgamation class. This expanded class,
enhanced with a linear ordering, has the Ramsey property.
Both forbidden trees and Ramsey properties have previously been linked 
to the complexity of constraint satisfaction problems.
\end{abstract}

\maketitle

\lineskiplimit -1pt

\section{Introduction}
\label{sec:intro}

Put vaguely, in Ramsey theory one looks for monochromatic subobjects
in colourings of large objects. For instance, one might want to prove
a statement like this:

\begin{em}
\begin{quotation}
\noindent
Let $A$, $B$ be digraphs and $r$ an integer.
Then there exists a digraph~$C$ such that whenever the copies of~$A$
in~$C$ are coloured with $r$~colours, then there exists a copy~$B'$ of~$B$
in~$C$ such that all the copies of~$A$ in~$B'$ have the same colour.
\end{quotation}
\end{em}

\noindent It is, however, not hard to show that this statement is false:
Let $A$ be a single arc and $B$ the directed 4-cycle.
Given any digraph~$C$, the adversary can colour the arcs of~$C$ with
$r\ge2$ colours as follows:
First, fix an arbitrary linear ordering of the vertex set of~$C$;
then colour every arc of~$C$ ``red'' if it goes ``forward'' with respect to
the ordering on its endpoints, and ``blue'' if it goes ``backward''.
Now, no matter how we order the vertices of~$B$ it will contain both a
forward and a backward arc~-- thus no copy of~$B$ in~$C$ can have all
its arcs in the same colour class.

This issue can be fixed by considering \emph{ordered} structures:
in this case we would consider digraphs with an additional linear ordering
of its vertices.
Then a ``forward'' arc and a ``backward'' arc are distinct, non-isomorphic
ordered digraphs and, in fact, the above statement becomes true.

\begin{thm}[Nešetřil--Rödl~\cite{NesRod:Partitions}]
\label{thm:11}
Let $A$, $B$ be ordered digraphs and $r$ an integer.
Then there exists an ordered digraph~$C$ such that whenever the copies of~$A$
in~$C$ are coloured with $r$~colours, then there exists a copy~$B'$ of~$B$
in~$C$ such that all the copies of~$A$ in~$B'$ have the same colour.
\end{thm}

The aim of this paper is to prove analogous results for $A$, $B$, $C$
belonging to specific classes of structures.
Perhaps the simplest example of such a result is the analogue of Theorem~\ref{thm:11} where
we replace ``ordered digraphs'' with ``ordered $K_n$-free undirected graphs'', proved
by Folkman~\cite{Fol:Graphs-with} for ${A=K_2}$, ${B=K_{n-1}}$,
as well as for $A=K_1$ and any $K_n$-free~$B$,
and by Nešetřil and Rödl~\cite{NesRod:Partitions,NesRod:Ramsey} in general.
Here we study classes of ordered digraphs (and, more generally, relational
structures) obtained not by forbidding one subgraph, such as the complete graph
in the example above, but by forbidding all homomorphic images of a
given set~$\F$ of oriented trees.
In this context a \emph{homomorphism} is a mapping that preserves the
arcs but it need not preserve the linear ordering.

Thus for a (possibly infinite) set~$\F$ of oriented trees,
let $\Forbh(\F)$ be the class of all ordered digraphs that admit no
homomorphism from any tree in~$\F$.
These classes are interesting in the context of constraint satisfaction problems.
For a finite digraph~$H$, $\Csp(H)$ denotes the class of all digraphs that admit a homomorphism to~$H$.
The case where $\Csp(H)=\Forbh(\F)$ for a set~$\F$ of trees corresponds to constraint
satisfaction problems with \emph{tree duality} (also known as
\emph{width-one} constraint satisfaction problems).
Ramsey theory provides a way to recognise digraphs~$H$ that define CSPs
with tree duality (see Section~\ref{sec:datalog}).

Now, however, another issue arises that can be illustrated with
this example:
Let $P_3$ be the directed path with three arcs and consider $\C=\Forbh(\F)$
for $\F=\{P_3\}$. For $A=P_1$, an arc ordered forward, and $B=P_2$, the
directed path with two arcs $0\to1\to2$ ordered $0\prec1\prec2$, there can
be no $P_3$-free~$C$ with the Ramsey property for $A$ and~$B$:
In any $C$ we can colour an arc ``red'' if there is another arc going out
from its head, and ``blue'' otherwise.
In any copy of $B=P_2$ in~$C$ the first arc will be red;
if there is a monochromatic copy~$B'$ of~$B$, then its second arc
must also be coloured red.
This implies, however, that there is a homomorphic image of~$P_3$ in~$C$.

The way to tackle this problem is to introduce new unary relations on
the vertices in a clever way (determined by the trees in~$\F$).
In our example, we would impose a unary relation on a vertex~$v$ (let us
call the unary relation ``square'') whenever there is an arc leaving~$v$,
and another unary relation (``circle'') whenever there is a copy of~$P_2$
leaving~$v$.
Then there is always both a square and a circle on the starting vertex
of~$P_2$, but never a circle on its middle vertex.
Hence the two arcs of~$P_2$ are no longer isomorphic induced subgraphs and we
never colour them both: $A$~cannot be an arc both with a circle on its
tail and without one.

The precise fashion in which the unary relations are introduced is described
in Section~\ref{sec:expand}.
Not always is it possible to use only finitely many unary relations.
Interestingly, it turns out that a finite number of unary relations
suffice if and only if $\Forbh(\F)=\Csp(H)$ for some finite~$H$.

The main result of this paper is the Ramsey property of any class
$\Forbh(\F)$ of ordered relational structures expanded by a number of
unary relations, with $\F$ being a set of relational trees.
The setting is properly defined in Section~\ref{sec:prelim};
the unary relations are introduced in Section~\ref{sec:expand}.
Section~\ref{sec:main} presents the main result,
which is then proved in Sections~\ref{sec:lemma} and~\ref{sec:constr}.
Section~\ref{sec:datalog} describes a link between constraint satisfaction problems
with tree duality and our Ramsey classes;
the paper then concludes with a number of final comments.


\section{Basic definitions}
\label{sec:prelim}

\paragraph{Relational structures.}

A \emph{signature}~$\sigma$ is a set of relation symbols; each
of the symbols has an associated \emph{arity}; the arity of~$R$ is~$\ar(R)$. A
\emph{$\sigma$-structure}~$A$ is a set of \emph{elements}, called the
\emph{domain} of~$A$ and denoted $\dom A$, together with a relation~$R^A$ of arity~$\ar(R)$
on the domain for every relation symbol~$R\in\sigma$.
Unless specifically stated otherwise, all structures we deal with
in this paper have finite domain. We allow the domain to be empty.
An \emph{ordered $\sigma$-structure} is a
$(\sigma\cup\{\preceq\})$-structure~$A$ such that ${\preceq}^A$~is
a linear ordering.
A \sigst~$A$ is a \emph{substructure}
of a \sigst~$B$ if $\dom A\subseteq \dom B$ and for each
$R\in\sigma$ we have $R^A=R^B\cap(\dom A)^{\ar(R)}$.
We write $A\subseteq B$ if $A$~is a substructure of~$B$.
Note that our substructure would be called an \emph{induced
substructure} in some literature.

An \emph{embedding}
of $A$ into~$B$ is a one-to-one mapping $f:\dom A\to\dom B$ such that for any
$R\in\sigma$ and any tuple~$\ol x$ we have $\ol x\in R^A$ iff $f(\ol
x)\in R^B$, where $f$~is applied on~$\ol x$ component-wise.
We write $f: A\eto B$ to indicate that $f$~is an embedding.
A bijective embedding is called an \emph{isomorphism}.
We write $A\cong B$ to indicate the existence of an isomorphism 
between $A$ and~$B$.

If $\sigma\subset\tau$, the \emph{$\sigma$-reduct} of a \taust~$A$ is the
\sigst~$A^\ast$ obtained from~$A$ by leaving out all the relations~$R^A$
for $R\in\tau\setminus\sigma$. Then $A$~is a \emph{$\tau$-expansion}
of~$A^\ast$.
(In some literature a reduct is called a \emph{shadow} and an
expansion is called a \emph{lift}.)

\paragraph{Ramsey classes.}

For any structures $A$, $B$, let $\binom{B}{A}$ denote the set of
all embeddings of~$A$ into~$B$.  The partition arrow $C\to(B)^A_r$
means that for any $\chi:\binom{C}{A}\to\{1,\dotsc,r\}$ (called
a \emph{colouring} with $r$ colours) there exists
$g\in\binom{C}{B}$ and $j\leq r$ such that $\chi(h)=j$ for all $h\in\binom{g[B]}{A}$.
In this case we call $g$ (or~$g[B]$) a \emph{monochromatic copy}
of~$B$ in~$C$.
Here $g[B]$ is the substructure of~$C$ induced by the range of~$g$,
which is isomorphic to~$B$ because $g$~is an embedding. Hence
$\binom{g[B]}{A}=\{g\circ f\colon f\in\binom{B}{A}\}$.

Let $\C$ be a class of finite structures and let $A\in\C$. The class~$\C$
has the \emph{$A$-Ramsey property} if for any $B\in\C$ and any
natural number~$r$ there exists $C\in\C$ such that $C\to(B)^A_r$.
The class~$\C$ is called a \emph{Ramsey class} if it has the
$A$-Ramsey property for all $A\in\C$.

\bigskip

Using the above terminology, Theorem~\ref{thm:11} is a special case of
the following:

\begin{thm}[Nešetřil--Rödl~\cite{NesRod:Partitions}]
\label{thm:nesrod}
Let $\sigma$ be a finite relational signature. Then the class of all
finite ordered \sigst s is a Ramsey class.
\end{thm}

The presence of orderings is indeed essential; cf.~the discussion
in~\cite{Nes:RamseyHom}.
Here we only note that for ordered structures $A,B$ there is a
one-to-one correspondence between the embeddings of~$A$ into~$B$
and substructures of~$B$ isomorphic to~$A$ because an ordered
structure has no non-trivial automorphisms.

\paragraph{Classes with forbidden homomorphic images.}

Let $A$, $B$ be \sigst s. A \emph{homomorphism} of~$A$ to~$B$ is a
mapping $f:\dom A\to\dom B$ such that for any $R\in\sigma$ and any
$\bar x\in R^A$ we have $f(\bar x)\in R^B$.
We write $f:A\hto B$ to indicate that $f$~is a homomorphism.

The interest of this paper lies in classes of finite \sigst s that can be
defined by forbidding the existence of a homomorphism from a given
set of structures. More explicitly, for a set~$\F$ of \sigst s let
$\Forbh(\F)$ be the class of all finite \sigst s~$A$ such that
whenever $F\in\F$, there exists no homomorphism of~$F$
to~$A$. In this case we say that \emph{$A$~is $\F$-free}.

\paragraph{Conventions.}\hfill

1.~A tuple has a bar, so $\bar x=(x_1,x_2,\dotsc,x_k)$ for some~$k$. If
$M$~is the domain of some function~$f$ and $\bar x\in M^k$, then $f(\bar
x)=(f(x_1),f(x_2),\dotsc,f(x_k))$.

2.~Instead of ``substructure of~$X$ generated by~$M$'' I write
``substructure of~$X$ induced by~$M$'' with the intended connotation
that the domain of such a substructure is actually~$M$.
This is the case because our structures have no operations.

3.~For a \sutst~$A$, $A^\ast$~almost always denotes the $\sigma$-reduct
of~$A$.

4.~Usually $R\in\sigma$ and $S\in\tau$, but sometimes $R\in\sigma\cup\tau$.

5.~I treat mappings and homomorphisms in a more set-theoretic rather
than category-the\-o\-retic way. For example, if $f:A\hto B$ is a
homomorphism and $B\subseteq C$, then also $f:A\hto C$. Similarly,
if $f:A\hto B$ is a homomorphism of \sutst s, then the same~$f$ is
a homomorphism of their $\sigma$-reducts; $f:A^\ast\hto B^\ast$.

\section{Amalgamation and other constructions}
\label{sec:amal}

A Ramsey class of structures always has the \emph{amalgamation property}
defined below (see~\cite{Nes:RamseyHom}).
Most classes of the form $\Forbh(\F)$ do
not have the amalgamation property, but following
Hu\-bič\-ka--Nešetřil~\cite{HubNes:Homomorphism,HubNes:Universal-structures}
there is a \emph{canonical way} to add new relations to the
signature~$\sigma$ in order to obtain it.
It is this \emph{expanded class}, enhanced with a linear ordering,
which is a Ramsey class.

\paragraph{Amalgamation.}

A class~$\C$ of finite \sigst s has the \emph{joint-embedding
property} if for any structures $A_1,A_2\in\C$ there exists $B\in\C$ such that
both $A_1$ and $A_2$ admit an embedding into~$B$.
A class~$\C$ of finite \sigst s has the \emph{amalgamation property}
if for any $A,B_1,B_2\in\C$ and any embeddings $f_1:A\eto B_1$ and
$f_2:A\eto B_2$ there exists $C\in\C$ and embeddings $g_1:B_1\eto C$
and $g_2:B_2\eto C$ such that $g_1\circ f_1=g_2\circ f_2$.
The amalgamation is \emph{free} if $\dom C=g_1[\dom B_1]\cup g_2[\dom B_2]$
and  $R^C=g_1[R^{B_1}]\cup g_2[R^{B_2}]$ for all $R\in\sigma$.
The amalgamation property implies the joint embedding property if
$\C$~contains the empty structure.

\paragraph{Sum.}

For two \sigst s $A$, $B$, their \emph{sum} $A+B$ is defined by
\begin{align*}
&\dom (A+B) = (\{A\}\times\dom A) \cup (\{B\}\times\dom B),\\
&R^{A+B} = (\{A\}\otimes R^A) \cup (\{B\}\otimes R^B),
\intertext{where}
&\{X\}\otimes R^X = \bigl\{\bigl((X,x_1),(X,x_2),\dotsc,(X,x_k)\bigr)\colon
	(x_1,x_2,\dotsc,x_k)\in R^X\bigr\}.
\end{align*}
Assuming that the domains of $A$ and~$B$ are disjoint, we could
take the union $\dom A\cup \dom B$ to be $\dom(A+B)$ and each
relation $R^{A+B}$ to be the union of the respective relations $R^A$
and~$R^B$. However, it will be convenient explicitly to mark which
summand an element of the sum originates from.
The definition can be extended to arbitrary finite sums in the obvious way.
We may also write $\coprod\{A_1,A_2,\dotsc,A_k\}$ for $A_1+A_2+\dotsb+A_k$.
Note that the sum of \sigst s is the coproduct in the category of
\sigst s and their homomorphisms.

Note further that if a class is closed under taking sums, then it
has the joint-embedding property. The converse is not true; for
instance, the class of complete graphs has the joint-embedding
property but is not closed under taking sums.

\paragraph{Connected structures.}
A \sigst~$A$ is \emph{connected} if, whenever $A\cong A_1+A_2$,
either $\dom A_1=\emptyset$ or $\dom A_2=\emptyset$. (This corresponds
to \emph{weak connectedness} of digraphs.)

\paragraph{Incidence graph.}
The \emph{incidence graph} $\Inc(A)$ of a \sigst~$A$ is the
bipartite undirected multigraph whose vertex set is $\dom
A\cup\bigcup\{R^A\times\{R\}:R\in\sigma\}$, and which contains for every
$R\in\sigma$, every $\bar x\in R^A$, and every~$i$, an edge joining
$(\bar x,R)$ and $x_i$.

\paragraph{Gaifman graph.}
The \emph{Gaifman graph} $\Gai(A)$ of a \sigst~$A$ is the graph
whose vertex set is $\dom A$ and there is an edge joining $x$ and~$y$
if $x,y$ are distinct elements of~$A$ that appear in a common tuple
of some relation of~$A$, that is,
\[E(\Gai(A)) = \bigl\{\{x,y\}\colon x\neq y \text{ and } \exists R\in\sigma\ \exists\bar v\in R^A\colon x,y\in\bar v\bigr\}.\]
Thus every tuple of every $R^A$ is represented by a clique in $\Gai(A)$.

\begin{lem}[see~\cite{F:Diss}]
\label{lem:connect}
For a \sigst~$A$, the following are equivalent:
\begin{enumerate}[label=\({\alph*}]
\item $A$ is connected.
\item Whenever $A\cong A_1 + A_2$, then $A\cong A_1$ or $A\cong A_2$.
\item $\Inc(A)$ is connected in the graph-theoretic sense.
\item $\Gai(A)$ is connected in the graph-theoretic sense.
\end{enumerate}
Moreover, if $A$ is connected and there is a homomorphism $f:A\hto
A_1+ A_2$, then either $f:A\hto A_1$ or $f:A\hto A_2$.
\qed
\end{lem}

\paragraph{Trees.}
A \sigst~$A$ is a \emph{$\sigma$-tree} (or just a \emph{tree}) if
$\Inc(A)$ is a tree.  (Thus in particular $A$ is \emph{not} a tree
if some tuple of some relation of~$A$ contains the same element 
more than once.)

\paragraph{Factor structure.}

If $A$ is a \sigst\ and $\sim$ is an equivalence relation on $\dom
A$, let the \emph{factor structure} $A/{\sim}$ be defined on $\dom
(A/{\sim}) = (\dom A)/{\sim}$ (the set of all equivalence classes
of~$\sim$) by letting $(X_1,X_2,\dotsc,X_k)\in R^{A/{\sim}}$ if and
only if there exist $x_1\in X_1$, $x_2\in X_2$,~\dots,~$x_k\in X_k$
such that $(x_1,x_2,\dotsc,x_k)\in R^{A}$.
Informally, $A/{\sim}$~is formed from~$A$ by identifying all the
elements in any one $\sim$-equivalence class into one element of~$A/{\sim}$.

\paragraph{Join of rooted structures.}
A \emph{rooted \sigst} is a couple $(A,a)$ where $A$~is a \sigst\
and $a\in\dom A$. Let $(A,a)$ and $(B,b)$ be rooted \sigst s. The
\emph{join} $(A,a)\oplus(B,b)$ is the factor structure $(A+B)/{\sim}$,
where $\sim$~is the equivalence relation on $\dom(A+B)$ such that
$(A,a)\sim(B,b)$, and $(X,x)\sim(X',x')$ iff $(X,x)=(X',x')$
otherwise. (Please do not get confused that, as a coincidence,
$(A,a)$ denotes a rooted structure on one occasion, and an element
of the sum $A+B$ on another.)
In other words, the join is obtained from the disjoint
union of $A$ and~$B$ by identifying $a$ and $b$.
In the obvious way, this definition can be extended to joins of
more than two structures, too.

\section{The expanded class}
\label{sec:expand}

From now on, we work in the following setting: $\sigma$~is a finite
relational signature; $\F$~is a possibly infinite set of finite
$\sigma$-trees. Recall that $\Forbh(\F)$ is the class of all \sigst
s that admit no homomorphism from any $F\in\F$.We are aiming to add
new unary relations, possibly infinitely many, to get a class~$\C$
of \sutst s such that
\begin{enumerate}
\item $\Forbh(\F)$ would be the class of the $\sigma$-reducts of
all structures in~$\C$;
\item $\C$ would have countably many isomorphism classes;
\item $\C$ would be closed under taking substructures (this is called the \emph{hereditary property});
\item $\C$ would have the amalgamation property;
\item $\C$ would have the Ramsey property.
\end{enumerate}
This can be achieved in a trivial way, by taking a new unary relation
for every element of every member of $\Forbh(\F)$ (only considering
one from each isomorphism class). We should therefore strive to
\begin{enumerate}
\setcounter{enumi}{5}
\item make $\tau$ ``as small as possible''; in particular, $\tau$~should
be finite whenever possible.
\end{enumerate}

\paragraph{Pieces of trees.}
Let $F\in\F$. An element $m$ of~$F$ is a \emph{cut} of~$F$ if it
is a vertex cut of $\Gai(F)$. Note that, as $F$~is a tree, $m$~is
a cut of~$F$ iff $m$~belongs to more than one tuple of the relations
of~$F$.

Let $m$ be a cut of~$F$ and let $D\subset\dom F$ be the vertex set
of some connected component of $\Gai(F)\setminus\{m\}$.  The rooted
\sigst\ $(M,m)$, where $M$~is the substructure of~$F$ induced by
$D\cup\{m\}$, is called a \emph{piece} of~$F$.

\begin{rems}
1.~A piece of $F$ is a non-empty connected substructure of~$F$,
$M\ne F$, and $\{m\}\ne\dom M$. Moreover, $M$~is a $\sigma$-tree.

2.~For any given cut $m$ of~$F$, the corresponding pieces cover
$\dom F$. In other words,
\[ F = \bigoplus\bigl\{(M,m)\colon (M,m) \text{ is a piece of } F\bigr\} \]
for any fixed cut $m$ of $F$.
\end{rems}

\paragraph{Equivalence of pieces.}

Let $(A,a)$ be a rooted \sigst. Following~\cite{HubNes:Universal-structures}, let
\[ \I(A,a) = \bigl\{(B,b)\colon \text{there exists $F\in\F$ s.t. } (A,a)\oplus(B,b)\cong F\bigr\} \]
be the set of all rooted \sigst s that are \emph{incompatible} with
$(A,a)$. For two pieces $(M,m)$ and $(M',m')$ we say that they are
\emph{equivalent} and write $(M,m)\pieq(M',m')$ if $\I(M,m)=\I(M',m')$.
Let $\P(\F)$ be the set of all $\pieq$-equivalence classes of all
pieces of all trees in~$\F$, that is,
\[\P(\F) = \bigl\{(M,m)\colon \exists F\in\F \text{ s.t.\ $(M,m)$ is a piece of } F\bigr\} / {\pieq} . \]

\begin{exa}
Oriented paths can be encoded by words over the alphabet $\{0,1\}$,
so that $0$~denotes a forward arc and $1$~denotes a backward arc.
Using this encoding, let $\F$ be the set of oriented paths encoded by the words
$000$, $00100$, $0010100$, $001010100$,~\dots\ (see Figure~\ref{fig:thunder}).
Let us call these paths \emph{thunderbolts}.
On the pieces of thunderbolts, the equivalence relation~$\pieq$ has
four equivalence classes, as shown in the figure.

\begin{figure}[ht]
\begin{center}
\includegraphics[width=4.8in]{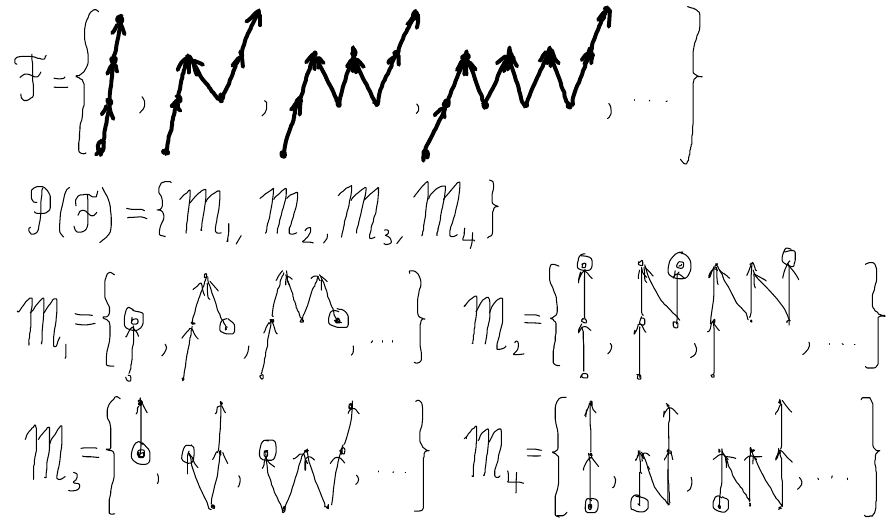}
\caption{Thunderbolts; all arcs are oriented upwards.}
\label{fig:thunder}
\end{center}
\end{figure}

For oriented paths, it can be shown that the number of $\pieq$-equivalence
classes is finite if the set~$\F$ can be described by a regular
language (as above): see~\cite{ErdTarTar:Caterpillar-dualities}.
\end{exa}

\paragraph{Subpieces.}
A \emph{subpiece} of a piece $(M,m)$ of~$F$ is any piece $(M',m')$
of~$F$ such that $M'$ is a substructure of~$M$.

\begin{lem}[\cite{HubNes:Universal-structures}]
\label{lem:subpiece}
Let $F_1\in\F$. Suppose that $(M_1,m_1)$ is a piece of $F_1$,
$(M_1',m_1')$ is a subpiece of $(M_1,m_1)$, and $(M_2',m_2')$ is
a piece such that $(M_1',m_1')\pieq(M_2',m_2')$. Create $(M_2,m_2)$
from $(M_1,m_1)$ by replacing the subpiece $(M_1',m_1')$ with
$(M_2',m_2')$, identifying~$m_1'$ with~$m_2'$; see
Figure~\ref{fig:subpiece}. Then $(M_2,m_2)$ is isomorphic to a piece
of some $F_2\in\F$ and $(M_1,m_1)\pieq(M_2,m_2)$.
\qed
\end{lem}

\begin{figure}[ht]
\begin{center}
\includegraphics[width=5in]{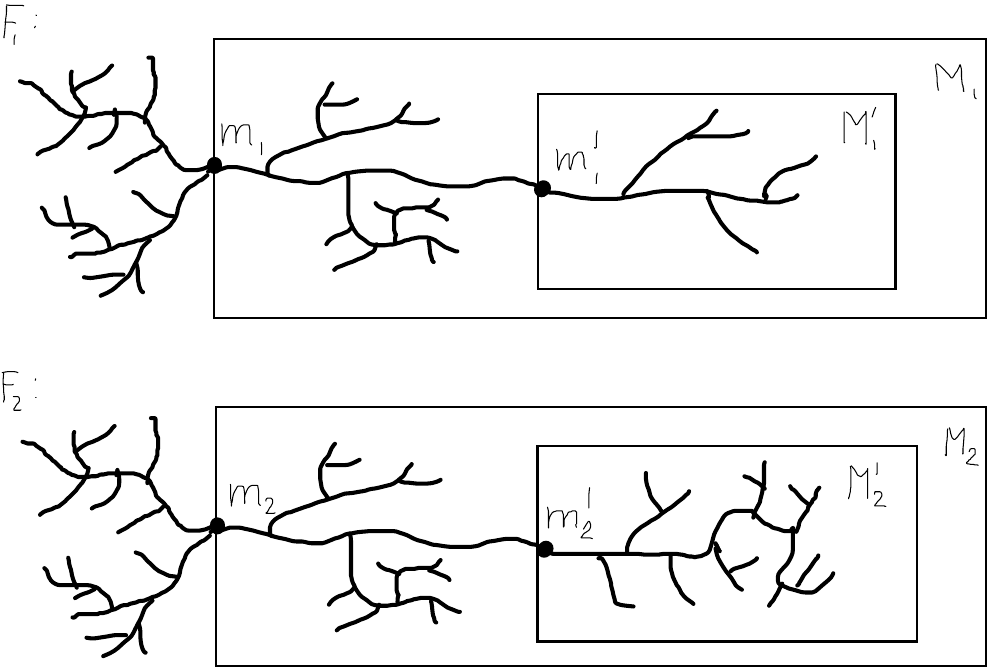}
\caption{Lemma~\ref{lem:subpiece}}
\label{fig:subpiece}
\end{center}
\end{figure}

\paragraph{Expansion.}
Now we define an expanded signature $\sigma\cup\tau$, aiming to get
a class~$\C$ of \sutst s with amalgamation, such that the
$\sigma$-reducts of all the structures in~$\C$ form exactly the
class~$\Forbh(\F)$.

\begin{defi}
\label{dfn:expanded}
Let $\tau$ contain a unary relation symbol~$S_\M$ for each 
$\M\in\P(\F)$. Let $\tilde\C$ be the class of finite
\sutst s such that $A$~belongs to~$\tilde\C$ if and only if the
$\sigma$-reduct~$A^\ast$ of~$A$ is in $\Forbh(\F)$ and for any 
$\M\in\P(\F)$ and any $x\in\dom A$ we have
\begin{equation}
\label{eq:canon}
x\in S_\M^A \qquad\iff\qquad \exists (M,m)\in\M,\ \exists f:M\hto A^\ast \text{ with }
f(m)=x.
\end{equation}
Let $\C$ be the class of all substructures of the structures in~$\tilde\C$.
The class~$\C$ is called the \emph{expanded class} for~$\Forbh(\F)$. The
structures in~$\tilde\C$ are called \emph{canonical}.
We can also say that \emph{$A$~is $\F$-free} if $A^\ast\in\Forbh(\F)$;
so being $\F$-free is a necessary but not sufficient condition for
membership in~$\C$.
\end{defi}

\begin{lem}
Every structure in~$\C$ satisfies the right-to-left implication
in~\eqref{eq:canon}.
\end{lem}

\begin{proof}
Let $A\in\C$, $x\in\dom A$, $(M,m)\in\M\in\P(\F)$, $f:M\hto A^\ast$
with $f(m)=x$. As $A\in\C$, $A$~is a substructure of some canonical
$\tilde A\in\tilde\C$. Then the same mapping~$f$ is a homomorphism
of~$M$ to~$\tilde A^\ast$. By definition, $\tilde
A$~satisfies~\eqref{eq:canon}, so $x\in S_\M^{\tilde A}$. Hence
$x\in S_\M^A$.
\end{proof}

\paragraph{Canonising.}

Given a \sutst~$A$, we want to find a
superstructure~$\tilde A$ of~$A$ that satisfies the left-to-right
implication of~\eqref{eq:canon}. This is possible assuming that
\begin{equation}
\label{eq:1elt}
\text{\emph{every one-element substructure of $A$ is in $\C$}.}
\end{equation}

\begin{lem}
\label{lem:canon}
If $A$ is a \sutst\ that satisfies \eqref{eq:1elt}, then there is
a \sutst~$\tilde A$ such that $A$~is a substructure of~$\tilde A$
and $\tilde A$~satisfies~\eqref{eq:1elt} as well as the left-to-right
implication of~\eqref{eq:canon}.
\end{lem}

\begin{proof}
For every $x\in\dom A$, let $A_x$ be the one-element substructure of~$A$ induced
by~$\{x\}$. By assumption~\eqref{eq:1elt}, for every~$x$ we have $A_x\in\C$; so
there exists $\tilde A_x\in\tilde\C$ containing~$A_x$. Let
\[A' = A + \textstyle\coprod\{\tilde A_x\colon x\in\dom A\} \]
and let $\sim$ be the smallest equivalence relation on $\dom A'$
such that $(A,x)\sim(\tilde A_x,x)$ for all $x\in\dom A$. Let $\tilde
A=A'/{\sim}$.
Informally, $\tilde A$~is obtained from~$A$ by gluing the
corresponding~$\tilde A_x$ on each element~$x$ of~$A$ (see
Figure~\ref{fig:tildA}).

\begin{figure}[ht]
\begin{center}
\includegraphics{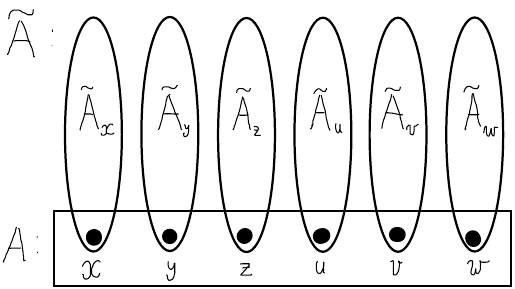}
\caption{$\tilde A$.}
\label{fig:tildA}
\end{center}
\end{figure}

Now each $\tilde A_x$ is isomorphic to a substructure of~$\tilde
A$. Moreover, any element of~$\tilde A$ is of the form $(\tilde
A_x,y)$ for some $x\in\dom A$ and some $y\in\dom\tilde A_x$.
Thus each one-element substructure of~$\tilde A$ is isomorphic
to a substructure of~$\tilde A_x\in\tilde\C$; hence it belongs to~$\C$.

Let $(\tilde A_x,y)\in S_\M^{\tilde A}$ for some $\M\in\P(\F)$. Then
$y\in S_\M^{\tilde A_x}$. Since $\tilde A_x\in\tilde\C$, there
exists $(M,m)\in\M$ and a homomorphism $f:M\hto\tilde A_x^\ast$
with $f(m)=y$. For any $n\in\dom M$, put $f'(n)=(\tilde A_x,f(n))$.
Then $f':M\hto \tilde A^\ast$ is a homomorphism with $f'(m)=(\tilde
A_x,f(m))=(\tilde A_x,y)$. Therefore $\tilde A$~satisfies the
left-to-right implication of~\eqref{eq:canon} as we wanted to show.
\end{proof}

\paragraph{Proving membership in $\C$.}

By definition, the class $\C$ (but not $\tilde\C$) is hereditary;
thus it is defined by a list of forbidden substructures. The goal
in the remainder of this section is to find an implicit description
of these forbidden substructures: we give a description in terms
of one-element and one-tuple structures. This description is going
to prove very useful in the sequel: it provides an interface between
the forbidden trees, pieces and unary relations on the one hand,
and the heavy machinery of Ramsey theory on the other.

\begin{lem}
\label{lem:one}
Let $E=E(F,m)$ be the \sutst\ obtained from some $F\in\F$ with a cut $m\in\dom F$ so that
\begin{align*}
&\dom E = \{1\};\\
&\text{for each } S_\M\in\tau,\quad 1\in S_\M^E \text{ iff\/ $\M$ contains a piece } (M,m) \text{ of } F;\\
&\text{all the $\sigma$-relations of $E$ are empty\/}.
\end{align*}
If $A$ is a \sutst\ such that there exists a homomorphism $f:E\hto A$, then $A\notin\C$.
\end{lem}

\begin{note}
In an informal way this lemma says that as soon as there are ``too
many'' $\tau$-relations on an element of a structure, this structure
is not in~$\C$. The forbidden substructures it describes are the
one-element structures with all the unary $\sigma$- and $\tau$-relations
corresponding to a fixed tree $F\in\F$ and a fixed cut~$m$ of~$F$,
as well as one-element structures with a superset of these relations.
\end{note}

\begin{proof}
Let $F\in\F$, let $m\in\dom F$ be a cut of~$F$, and consider
$E=E(F,m)$. Let the corresponding pieces of~$F$ be $(M_1,m)$,
$(M_2,m)$,~\dots,~$(M_k,m)$.

Let $A$ be a \sutst.  For the sake of contradiction, suppose that
there is a homomorphism $f:E\hto A$ but $A\in\C$. Then there
is a canonical superstructure $\tilde A\in\tilde\C$ of~$A$.

For each $i=1,2,\dotsc,k$ let $\M_i\in\P(\F)$ be the $\pieq$-equivalence
class of the piece $(M_i,m)$. By definition, $1\in S^E_{\M_i}$. As
$f$~is a homomorphism, we have $f(1)\in S^A_{\M_i}\subseteq S^{\tilde
A}_{\M_i}$. Since $\tilde A$~is canonical, by~\eqref{eq:canon} there
exists a piece $(N_i,n_i)\in\M_i$ and a homomorphism $g_i: N_i\hto
\tilde A^\ast$ with $g_i(n_i)=f(1)$. The union of all the
homomorphisms~$g_i$, $i=1,\dotsc,k$, is a homomorphism~$g$ of
$(N_1,n_1)\oplus\dotsb\oplus(N_k,n_k)$ to~$\tilde A^\ast$.

Now let \[F_0=(M_1,m)\oplus\dotsb\oplus(M_k,m)\cong F.\]
For $i\ge1$, put
\[F_i=(N_1,n_1)\oplus(N_2,n_2)\oplus\dotsb\oplus(N_i,n_i)
\oplus(M_{i+1},m)\oplus\dotsb\oplus(M_k,m).\]
We prove by induction on~$i$ that each $F_i$ is isomorphic to a
member of~$\F$. In the rest of this proof, let ``$\in\F$'' mean
``is isomorphic to a member of~$\F$'', to simplify the notation.
Clearly, $F_0\in\F$ because $F_0\cong F$. For $i\ge1$, let
\[\bar M_i=(N_1,n_1)\oplus\dotsb\oplus(N_{i-1},n_{i-1})
\oplus (M_{i+1},m)\oplus\dotsb\oplus(M_k,m),\]
whence $F_{i-1}=(\bar M_i,m)\oplus(M_i,m)$ and $F_i=(\bar
M_i,m)\oplus(N_i,n_i)$.  Assuming that $F_{i-1}\in\F$, we thus get
that $(\bar M_i,m)\in\I(M_i,m)=\I(N_i,n_i)$ because
$(M_i,m)\pieq(N_i,n_i)$. Hence $F_i\in\F$.

We conclude that $F_k\in\F$ and $g:F_k\hto\tilde A^\ast$ is a homomorphism;
hence $\tilde A$~is not $\F$-free~-- a contradiction with the
assumption that $\tilde A\in\tilde\C$.
\end{proof}

\paragraph{Tuple traces.}

Let $A$ be a \sutst, and $R\in\sigma$.  The \emph{tuple trace} of some $\bar
x=(x_1,x_2,\dotsc,x_k)\in R^A$ is the structure~$T=T(A,\bar x,R)$
with $\dom T=\{1,2,\dotsc,k\}$; $R^T=\{(1,2,\dotsc,k)\}$;
$\check R^T=\{j\colon x_j\in\check R^A\}$ for all unary $\check R\in\sigma$;
$R'^T=\emptyset$ for any other $R'\in\sigma\setminus\{R\}$; $S^T=\{j\colon
\allowbreak{x_j\in S^A}\}$ for $S\in\sigma$.

\begin{exa}
For $\sigma$ containing a quaternary relation symbol~$R$ and a binary~$R'$,
and $\tau$ containing two unary relation symbols $\diamdot$ and $\odot$, 
let $A$ be the structure with $\dom A=\{a,b,c\}$, $R^A=\{(a,b,b,c)\}$,
$R'^A=\{(a,c)\}$, $\diamdot^A=\{a,b\}$ and $\odot^A=\{c\}$. The tuple
trace of $(a,b,b,c)\in R^A$ is the structure~$T$ with $\dom
T=\{1,2,3,4\}$, $R^T=\{(1,2,3,4)\}$, $R'^T=\emptyset$,
$\diamdot^T=\{1,2,3\}$, $\odot^T=\{4\}$ (see Figure~\ref{fig:tutrace}).
\end{exa}

\begin{figure}[ht]
\begin{center}
\includegraphics[width=4.3in]{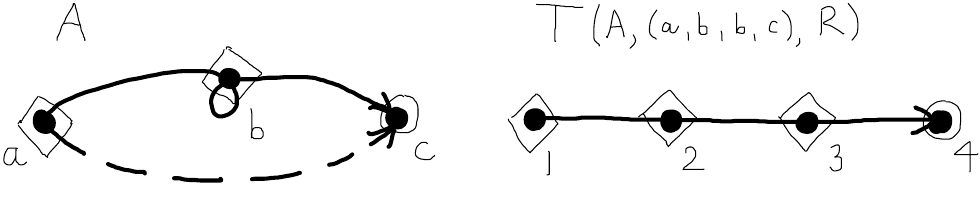}
\caption{Tuple trace for the example just before Lemma~\ref{lem:crit}.
Think of ``unfolding'' the tuple $\bar x = (a,b,b,c)\in R^A$, while
keeping all the unary relations but dropping all the non-unary ones.}
\label{fig:tutrace}
\end{center}
\end{figure}

\begin{lem}
\label{lem:crit}
Let $A$ be a \sutst. Then $A\in\C$
if and only if each one-element substructure of~$A$
belongs to~$\C$, and for any $R\in\sigma$ and any $\bar
x\in R^A$, the tuple trace of~$\bar x$ belongs to~$\C$.
\end{lem}

\begin{proof}
If $A\in\C$, then each one-element substructure of~$A$ is in~$\C$
as well because $\C$~is hereditary. Let
$A\subseteq\tilde A\in\tilde\C$. Consider any $R\in\sigma$ and $\bar
x=(x_1,\dotsc,x_k)\in R^A\subseteq R^{\tilde A}$. Let $T=T(A,\bar
x,R)$ be the tuple trace of $\bar x\in R^A$, which is, in fact,
equal to $T(\tilde A,\bar x,R)$. Let $T'$ be the sum of $T$ and $k$
copies of~$\tilde A$; let $\sim$ be the smallest equivalence
relation that identifies $j\in\dom T$ with $x_j$~in the $j$th copy
of~$\tilde A$. Let $\tilde T = T'/{\sim}$. There is an obvious
``projection'' or ``folding'' homomorphism $p:\tilde T\hto\tilde A$; the image
under~$p$ of~$T$ is the substructure of~$\tilde A$ induced by $\{x_1,x_2,\dotsc,x_k\}$
and $p$ restricted to any of the $k$~copies of~$\tilde A$ in~$\tilde T$ is an isomorphism.
Note that the same mapping~$p$ is a homomorphism of the $\sigma$-reduct
$\tilde T^\ast$ to~$\tilde A^\ast$.
See Figure~\ref{fig:tutrace2}.

\begin{figure}[ht]
\begin{center}
\includegraphics[width=4.8in]{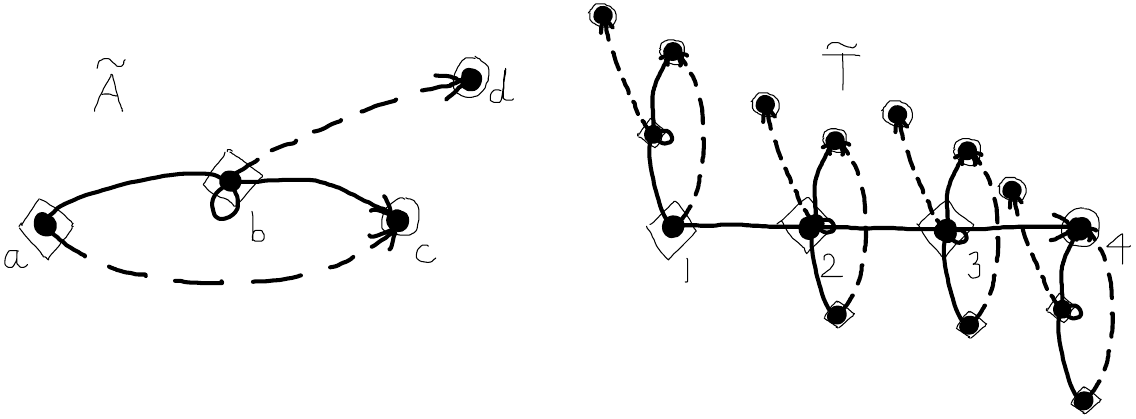}
\caption{$\tilde T$ for $\bar x\in R^A$ from Figure~\ref{fig:tutrace}.
$\tilde T$~is obtained from the tuple trace~$T$ by gluing a copy
of~$\tilde A$ on each element of~$T$.}
\label{fig:tutrace2}
\end{center}
\end{figure}

If $F\in\F$ and $f:F\hto\tilde T^\ast$, then $p\circ f:F\hto\tilde A^\ast$, a
contradiction with $A\in\C$. Thus $\tilde T$~is $\F$-free.
To show that $\tilde T$~satisfies~\eqref{eq:canon}, first let
$(M,m)\in\M\in\P(\F)$ and let $g:M\hto \tilde T^\ast$. Since $\tilde A$
satisfies~\eqref{eq:canon} and $p\circ g:M\hto\tilde A^\ast$ is a homomorphism,
we have $p(g(m))\in S_{\M}^{\tilde A}$.  Hence, by the definition
of~$\tilde T$, we have $g(m)\in S_{\M}^{\tilde T}$. Conversely, if
$x\in\dom\tilde T$ satisfies $x\in S_{\M}^{\tilde T}$ for some
$\M\in\P(\F)$, then $p(x)\in S_\M^{\tilde A}$, thus there exist
$(M,m)\in\M$ and a homomorphism $h:M\hto\tilde A^\ast$ such that $h(m)=p(x)$.
Mapping each element~$a$ of~$M$ to the element corresponding
to~$h(a)$ in the copy of~$\tilde A$ within~$\tilde T$ that contains~$x$
provides a homomorphism from~$M$ to~$\tilde T$ that takes~$m$ to~$x$.
Therefore not only $\tilde T$ is $\F$-free but it satisfies~\eqref{eq:canon} as well,
so $\tilde T\in\tilde\C$. The tuple trace~$T$, which is a substructure
of~$\tilde T$, then belongs to~$\C$.

The converse implication: Suppose that $A$~satisfies~\eqref{eq:1elt}
and all its tuple traces belong to~$\C$.  By Lemma~\ref{lem:canon} there
is a \sutst~$\tilde A$ such that $A$~is a substructure of~$\tilde
A$ and $\tilde A$~satisfies~\eqref{eq:1elt} and the left-to-right
implication of~\eqref{eq:canon}.  Observe that any tuple trace
of~$\tilde A$, as described in the proof of Lemma~\ref{lem:canon},
is equal to a tuple trace of~$A$ (hence in~$\C$ by assumption)
or to a tuple trace of some~$\tilde A_x\in\tilde \C$ (hence in~$\C$ by the first
implication of this lemma, which we have just proved). Thus we
may assume that the tuple traces of $\tilde A$ belong to~$\C$.

To prove that $\tilde A$~satisfies the right-to-left implication
of~\eqref{eq:canon}, let $\tilde A^\ast$ be the $\sigma$-reduct
of~$\tilde A$, let $\M\in\P(\F)$ and let $(M,m)\in\M$ be a piece
of some $F\in\F$ and consider any homomorphism $f:M\hto\tilde A^\ast$.
We want to show that $f(m)\in S_\M^{\tilde A}$. For the sake of
contradiction, assume that $f(m)\notin S_\M^{\tilde A}$ and that $M$~is
a minimal such piece, that is, we assume that whenever $(N,n)$ is
a subpiece of~$(M,m)$, $(N,n)\in\N\in\P(\F)$, then $f'(n)\in
S_{\N}^{\tilde A}$ for any homomorphism $f':N\hto\tilde A^\ast$.

Because $(M,m)$ is a piece, $m$~belongs to a unique tuple $\bar
x\in R^M$ with $k=\ar(R)>1$. Let $(N_1,n_1)$,
$(N_2,n_2)$,~\dots,~$(N_\ell,n_\ell)$ be all the pieces of~$F$
corresponding to all cuts~$x_i$ of~$F$, $x_i\ne m$, such that
$m\notin\dom N_j$ for any~$j$. Each $(N_j,n_j)$ is a subpiece of
$(M,m)$ and each $n_j\in\bar x$ (see Figure~\ref{fig:nj}). Let
$\iota(j)$ be the index for which $n_j=x_{\iota(j)}$, and let
$\iota(0)$ be the index for which $m=x_{\iota(0)}$. Moreover, let
$\N_j\in\P(\F)$ be the $\pieq$-equivalence class of~$(N_j,n_j)$.

\begin{figure}[ht]
\begin{center}
\includegraphics[width=5in]{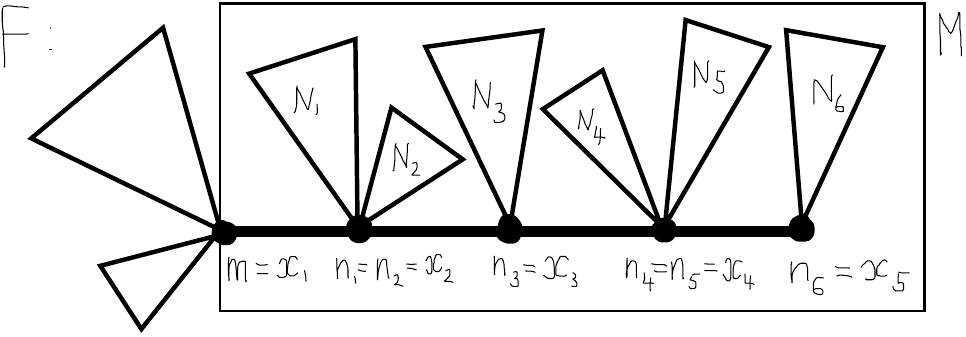}
\caption{In this example, $\ell=6$, $k=5$.}
\label{fig:nj}
\end{center}
\end{figure}

Consider the tuple trace $T=T(\tilde A,f(\bar x),R)$ of $f(\bar
x)\in R^{\tilde A}$.  As we have observed, $T\in\C$, hence there
exists canonical $\tilde T\in\tilde\C$ such that $T$~is a substructure
of~$\tilde T$. By definition, $\dom T=\{1,2,\dotsc,k\}\subseteq\dom\tilde
T$, and for each~$j$ we have $i\in S^{\tilde T}_{\N_j}$ if and only
if $f(x_i)\in S^{\tilde T}_{\N_j}$. By minimality of counterexample,
however, we have $\iota(j)\in S_{\N_j}^{\tilde T}$ for each
$j\in\{1,\dotsc,\ell\}$, as each $f_j=f\restriction N_j:N_j\hto\tilde
A^\ast$ is a homomorphism with $f_j(n_j)=f(x_{\iota(j)})$.  Because
$\tilde T$~is canonical, that is, it satisfies~\eqref{eq:canon},
for each~$j$ there exists a homomorphism~$g_j$ of some~$N_j'$
to~$\tilde T^\ast$ with $(N_j',n_j')\in\N_j$ and $g_j(n_j')=\iota(j)$.

Let $(M',m)$ be obtained from~$(M,m)$ by replacing each subpiece
$(N_j,n_j)$ with the piece~$(N_j',n_j')$. By Lemma~\ref{lem:subpiece},
$(M',m)$~is a piece and $(M',m)\pieq(M,m)$, thus $(M',m)\in\M$. Let
$g:\dom M'\to\dom\tilde T^\ast$ be defined by
\begin{alignat*}{2}
g(x_i) &= i &\quad& \text{for all } i = 1,\dotsc,k\,; \\
g(u)  &= g_j(u) && \text{for each } u\in\dom N_j' \text{ s.t.~} u\notin\bar x\,.
\end{alignat*}
Clearly $g:M'\hto\tilde T^\ast$ is a homomorphism with $g(m)=\iota(0)$.
Hence $\iota(0)\in S_\M^{\tilde T}$ and also $\iota(0)\in S^T_\M$.
By the definition of a tuple trace, $f(m)=f(x_{\iota(0)})\in S^{\tilde
A}_\M$, a contradiction.

Thus we have shown that $\tilde A$~satisfies~\eqref{eq:canon}.  Next
we show that $\tilde A$ is $\F$-free.  Suppose there is some $F\in\F$
and a homomorphism $f:F\hto\tilde A^\ast$.  If $F$~has only one
element, then the one-element substructure of~$\tilde A$ induced
by~$f[F]$ is not $\F$-free, hence
not in~$\C$, a contradiction. If $F$~has more than one element but
it is irreducible (that is, if it contains exactly one tuple~$\bar
x$ of a relation~$R^F$ of arity more than one), then the tuple trace
of~$f(\bar x)\in R^{\tilde A}$ is not in~$\C$, again a contradiction.
Hence $F$~has a cut~$m$. Also, for any piece $(M,m)\in\M$ of~$F$
the restriction $g=f\restriction M$ is a homomorphism $g:M\hto \tilde
A^\ast$ such that $g(m)=f(m)$. Thus $f(m)\in S_\M^A$ for any such
piece $(M,m)\in\M$.  Let $E=E(F,m)$ be obtained from~$F$ with the
cut~$m$ as in Lemma~\ref{lem:one}.  Then the one-element substructure
of~$\tilde A$ induced by~$\{f(m)\}$ admits a homomorphism from~$E$,
so by Lemma~\ref{lem:one} it is not in~$\C$, once again a contradiction.
We conclude that $\tilde A^\ast\in\Forbh(\F)$.

Therefore $\tilde A\in\tilde\C$, and so $A\in\C$.
\end{proof}

Next is the amalgamation property of~$\C$. The following theorem
is proved in~\cite{HubNes:Universal-structures} for the case of
$\tau$ being finite, but it does not require $\F$ to contain only trees
(the arities of $\tau$-relations will then in general be greater
than one, and the amalgamation will not be free).

\begin{thm}
\label{thm:amalgamation}
Let $\sigma$ be a finite relational signature,
let $\F$ be a set of finite $\sigma$-trees and let $\C$ be the
expanded class for $\Forbh(\F)$. Then
\begin{enumerate}
\item the class of all $\sigma$-reducts of the structures in~$\C$
	is $\Forbh(\F)$;
\item $\C$~is closed under isomorphism;
\item $\C$~is closed under taking substructures;
\item $\C$~has only countably many isomorphism classes;
\item $\C$~has the free amalgamation property.
\end{enumerate}
\end{thm}

\begin{proof}
(1), (2), (3) follow immediately from Definition~\ref{dfn:expanded}
and from the fact that $\Forbh(\F)$ is closed under both isomorphism
and taking substructures. The class $\Forbh(\F)$ has only countably
many isomorphism classes because it is a class of finite \sigst s
over a finite relational signature~$\sigma$. The canonical
class~$\tilde\C$ contains exactly one structure for each structure
in~$\Forbh(\F)$ (the one given by~\eqref{eq:canon}), thus $\tilde\C$~also
has only countably many isomorphism classes.  Finally, $\C$~contains
finitely many structures for each structure in~$\tilde\C$, hence
$\C$~has only countably many isomorphism classes.

To prove (5), let $A,B_1,B_2\in\C$ and let $f_1:A\eto B$, $f_2:A_2\eto
B$ be embeddings. Without loss of generality we may assume that
$f_1,f_2$ are inclusion mappings, that is, $\dom A\subseteq\dom
B_1$ and $\dom A\subseteq\dom B_2$, and that $\dom B_1\cap \dom
B_2=\dom A$. Let $C$ be the \sutst\ defined as follows:
\begin{alignat*}{2}
\dom C &= \dom B_1\cup\dom B_2\,,\\
R^C &= R^{B_1} \cup R^{B_2} &\quad&\text{for any } R\in\sigma\cup\tau\,.
\end{alignat*}
Now, every one-element substructure of~$C$ is a substructure of
either $B_1$ or $B_2$ (or both), thus by Lemma~\ref{lem:crit}
$C$~satisfies~\eqref{eq:1elt}.
Moreover, whenever $\bar x\in R^C$ for some $R\in\sigma$, then $\bar
x\in R^{B_1}$ or $\bar x\in R^{B_2}$, hence by the same lemma the
tuple trace of~$\bar x\in R^C$ belongs to~$\C$. Using the converse
implication of Lemma~\ref{lem:crit} we get that $C\in\C$. It is
easy to see that the inclusion mappings are embeddings of~$B_1$
and~$B_2$ to~$C$.
\end{proof}

After all this preparation, we are ready for the main result,
presented in the next section.

\section{Main result}
\label{sec:main}

\paragraph{Orderings.}
Recall that an \emph{ordered $\upsilon$-structure} is
a $(\upsilon\cup\{\preceq\})$-structure~$A$ such that the
relation~${\preceq}^A$ is a linear ordering.

\begin{defi}
\label{dfn:3.1}
Let $\sigma$ be a finite relational signature and
let $\F$ be a set of finite $\sigma$-trees. The \emph{ordered expanded
class} for $\Forbh(\F)$ is the class~$\vec\C$ of ordered \sutst s such
that $A\in\vec\C$ if and only if ${\preceq}^A$ is a linear ordering and
the $(\sigma\cup\tau)$-reduct of~$A$ is in the expanded class~$\C$
for $\Forbh(\F)$.
\end{defi}

\begin{note}
As a consequence of Theorem~\ref{thm:amalgamation}, the class~$\vec\C$
is closed under isomorphism and taking substructures. It also has
the amalgamation property: Take the amalgam of the
$(\sigma\cup\tau)$-reducts; the union of the orders $\preceq^{B_1}$
and $\preceq^{B_2}$ is a reflexive anti-symmetric relation on~$\dom
C$ whose transitive closure is a partial ordering, and any of its
linear extensions can be taken as~$\preceq^C$.
\end{note}

\begin{thm}
\label{thm:main}
Let $\sigma$ be a finite relational signature
and let $\F$ be a set of finite $\sigma$-trees.
Then the ordered expanded class for\/ $\Forbh(\F)$ has the Ramsey property.
\end{thm}

\begin{rem}
It has recently been announced by Nešetřil~\cite{Nes:CSS-Ramsey} that
the ordered expanded class is a Ramsey class if $\F$ is a \emph{finite}
set of finite connected \sigst s.
Our Theorem~\ref{thm:main} allows \emph{infinite}~$\F$, but requires that
all elements of~$\F$ are \emph{$\sigma$-trees}.
\end{rem}

\paragraph{Idea of proof.}

The proof of this theorem, which spreads over the following two
sections, is based on the ideas of \emph{partite lemma} and
\emph{partite construction}, developed by Nešetřil and
Rödl \cite{NesRod:Simple,NesRod:Ramsey,NesRod:Strong,NesRod:The-partite}.

The principal idea is the notion of a \emph{partite structure}. A
partite structure is an ordered \sutst\ whose domain is split into
several parts. The parts of a partite structure~$X$ are indexed by
elements of some ordered \sigst~$P$; formally, the part of
an element is determined by a mapping $\iota_X:\dom X\to\dom P$.
Furthermore, we want $\iota_X$ to be a homomorphism of the
\sigord-reduct of~$X$ to~$P$. Informally, we want the tuples of the
$\sigma$-relations only to sit across those parts of~$X$ where $P$
also has a corresponding tuple for the same relation symbol.
The ordering~$\preceq^A$ preserves the ordering of parts given
by~$\preceq^P$; within one part, the ordering can be arbitrary.

Another desired property of partite structures is somewhat peculiar:
We do not want a tuple~$\bar x$ of a $\sigma$-relation of~$X$ to
contain two different elements from the same part. This applies if
a tuple of~$P$ contains an element of~$P$ more than
once. For instance, let $R\in\sigma$ and $(a,b,b)\in R^P$. Then we
allow a tuple $(x,y,y)\in R^X$ if $\iota_X(x)=a$ and $\iota_X(y)=b$,
but we do not allow $(x,y,y')\in R^X$ for $y\ne y'$. Formally, we require
$\iota_X$ to be injective on the set of elements of any tuple of a relation of~$X$. 

Finally, we want all the elements of any part to belong to exactly
the same unary $\sigma$-relations as the corresponding element
of~$P$ (the $\tau$-relations are not prescribed by~$P$, which is a
\sigst, and can vary within one part of a partite structure). In addition,
if there is a ``loop'' $(a,a,\dotsc,a)\in R^P$ for some $R\in\sigma$,
we want that $(x,x,\dotsc,x)\in R^X$ for each~$x$ with $\iota_X(x)=a$.

\paragraph{Partite lemma.}

The partite lemma is often proved by an
application of the Hales--Jewett theorem (as
in~\cite{NesRod:Strong,NesRod:The-partite,NesetrilHandbookChapter}).
Our proof, however, is inspired by that of Prömel and
Voigt~\cite{ProVoi:A-short}.
We prove the Ramsey property for a class of
very special structures: \emph{rectified structures}. A rectified
structure is similar to a partite structure in that its domain is
split into parts, indexed this time by the elements of some
$A\in\vec\C$. Furthermore, if we choose one element from each part,
the induced substructure of~$X$ is isomorphic to~$A$.  Conversely,
any tuple of a relation of~$X$ lies within some such copy of~$A$
in~$X$. Thus a rectified structure~$X$ is actually fully determined
(up to isomorphism) by~$A$ and the size of each part $\iota^{-1}[a]$
of~$X$, $a\in\dom A$ (see Figure~\ref{fig:l52}). In this case,
the proof of the Ramsey property is relatively straightforward by
induction on the number of elements of~$A$ (i.e., the number of
parts), and the base step as well as each induction step follow
from the pigeon-hole principle.

\paragraph{Partite construction.}

The partite lemma is then applied repeatedly in the partite
construction (also known as the \emph{amalgamation method}). The
idea is as follows: Given $A,B\in\vec\C$, we want to find $C\in\vec\C$
such that $C\to(B)^A_r$. By Theorem~\ref{thm:nesrod} we know that
there exists an ordered \sutst~$C$ with $C\to(B)^A_r$, but there is no
guarantee that $C\in\vec\C$. Indeed, typically such~$C$ will be far
from being $\F$-free and satisfying~\eqref{eq:canon}. However, it
will be a good starting point. In fact, we will use a \sigst~$P$
such that $P\to(B^\ast)^{A^\ast}_r$, where $A^\ast$, $B^\ast$ are
the (ordered) $\sigma$-reducts of $A$,~$B$. This~$P$ serves as the
indexing structure for the parts of our partite structures.

The partite construction then works inductively, starting from~$C_0$,
which is basically
a suitable sum (disjoint union) of partite structures isomorphic
to~$B$. In each induction step (one for every occurrence of a copy
of~$A^\ast$ in~$P$), it uses the partite lemma for rectified
structures. To this end, each application of the partite lemma must
be preceded by adding new tuples to the relations in order to make
the structure rectified. This construction is called \emph{rectification};
we can show that this can be done without losing membership in~$\vec\C$
(Lemma~\ref{lem:rectification}).

The inductive construction produces a structure~$C$ with the following
property: Whenever $\binom{C}{A}$ is $r$-coloured, there exists a
copy~$C_0'$ of~$C_0$ within~$C$ such that the colour of each copy~$A'$
of~$A$ within~$C_0'$ depends solely on the parts this copy sits on
(i.e., on the image $\iota_C[A'])$. Finally, the way we construct~$C_0$
and Theorem~\ref{thm:nesrod} then guarantee the existence of a
monochromatic copy of~$B$ within this~$C_0'$.

\section{Partite lemma}
\label{sec:lemma}

Throughout this section, $\F$ is a fixed set of finite $\sigma$-trees and
$\vec\C$~is the ordered expanded class for $\Forbh(\F)$.

\paragraph{Rectified structures.}

Let $A\in\vec\C$. An \emph{$A$-rectified structure} is a pair
$(X,\iota_X)$ such that $X$~is an ordered \sutst, $\iota_X:\dom
X\to\dom A$ is a mapping, ${x\preceq^X x'}$ implies that
$\iota_X(x)\preceq^A\iota_X(x')$, and for any $R\in\sigma\cup\tau$
and any $\bar x\in(\dom X)^{\ar(R)}$ we have
\begin{equation}
\label{eq:recti}
\bar x\in R^X\qquad \iff\qquad \text{$\iota_X$ is injective on $\bar x$ and }
\iota_X(\bar x)\in R^A.
\end{equation}
(The mapping $\iota_X$ is injective on $\bar x=(x_1,x_2,\dotsc,x_k)$
if $x_i=x_j$ whenever $\iota_X(x_i)=\iota_X(x_j)$.)
Observe that $X$~is uniquely determined by $A$, $\dom X$ and $\iota_X$
via~\eqref{eq:recti}.

A mapping $e:\dom X\to\dom Y$ is an embedding of $A$-rectified structure
$(X,\iota_X)$ into $(Y,\iota_Y)$ if $e:X\eto Y$ is an embedding of \sutordst
s and $\iota_X=\iota_Y \circ e$.

\begin{prop}
\label{prop:recti}
Let $A\in\vec\C$.
\begin{enumerate}
\item\label{rec-1}
  If $(X,\iota_X)$ is $A$-rectified, then $X\in\vec\C$.
\item\label{rec-2}
  If $(X,\iota_X)$ is $A$-rectified, then the mapping $\iota_X$
  is a homomorphism of~$X$ to~$A$.
\item\label{rec-3}
  $(A,\id_A)$ is $A$-rectified.
\item\label{rec-4}
  For any $A$-rectified $(X,\iota_X)$, any mapping $e:\dom
  A\to\dom X$ such that $\iota_X\circ e=\id_A$ is an embedding of~$A$
  into~$X$, as well as an embedding of $(A,\id_A)$ into $(X,\iota_X)$.
\item\label{rec-5}
  If $\abs{\dom A}=1$ and $(X,\iota_X)$ is an $A$-rectified structure,
  then $X$ is the disjoint union of several (possibly 0 or 1) copies of~$A$,
  endowed with some linear ordering of its elements, and $\iota_X$~is
  constant.  Conversely, any such $(X,\iota_X)$ is $A$-rectified.
\end{enumerate}
\end{prop}

\begin{proof}
\leavevmode\par

(\ref{rec-1}) Let $\hat A,\hat X$ be the $(\sigma\cup\tau)$-reducts
of $A,X$, respectively. We apply Lemma~\ref{lem:crit}. Let $x\in\dom
\hat X$. By~\eqref{eq:recti}, the one-element substructure~$\hat X_x$
of~$\hat X$ induced by~$x$ is isomorphic to the substructure
of~$\hat A$ induced by~$\iota_X(x)$. Thus $\hat X_x\in\C$ because
$\hat A\in\C$. Next, for any $R\in\sigma$ and $\bar x\in R^{\hat X}$,
the tuple trace $T=T(\hat X,\bar x,R)$ is equal to the tuple trace
$T(\hat A,\iota_X(\bar x),R)$ of~$\iota_X(\bar x)\in R^{\hat A}$
because of~\eqref{eq:recti}.  Hence $T\in\C$. By Lemma~\ref{lem:crit},
$\hat X\in\C$. Therefore $X\in\vec\C$.

(\ref{rec-2}) All $\sigma$- and $\tau$-relations are preserved by~\eqref{eq:recti},
and $\preceq$~is preserved by definition.

(\ref{rec-3}) Obvious.

(\ref{rec-4}) Since $\iota_X\circ e$ is injective, $e$~is injective
and $\iota_X$~is injective on any tuple of~$e[A]$. Let
$R\in\sigma\cup\tau$. If $\bar a\in R^A$, then $\iota_X(e(\bar
a))=\bar a$, so by~\eqref{eq:recti}, $e(\bar a)\in R^X$. Conversely,
if $\bar a\in(\dom A)^{\ar(R)}$ and $e(\bar a)\in R^X$, then $\bar
a=\iota_X(e(\bar a))\in R^A$ because $\iota_X$~is a homomorphism.
An analogous argument applies to preservation of~$\preceq$.

(\ref{rec-5}) If $\iota_X:\dom X\to\dom A$ and $\abs{\dom A}=1$, then certainly
$\iota_X$~is constant. But then if $\bar x\in R^X$ for some
$R\in\sigma\cup\tau$, then $\bar x$~is of the form $(x,x,\dotsc,x)$
because $\iota_X$~is injective on~$\bar x$. Hence the only relation
of~$X$ that spans distinct elements of~$X$ is the linear
ordering~$\preceq^X$. It follows from~\eqref{eq:recti} that each
one-element substructure of~$X$ is isomorphic to~$A$.

The converse is obvious.
\end{proof}

\begin{lem}[Partite Lemma]
\label{lem:partite}
Let $\F$ be a set of finite $\sigma$-trees and let $\vec\C$ be
the ordered expanded class for $\Forbh(\F)$; let $A\in\vec\C$.
Let $(B,\iota_B)$ be $A$-rectified; let $r\geq1$. Then
there exists $A$-rectified $(E,\iota_E)$ such that
$(E,\iota_E)\to(B,\iota_B)^{(A,\id_A)}_r$.
\end{lem}

\begin{proof}
By induction on $\abs{\dom A}$. If $\abs{\dom A}=1$, take $E$ to be the disjoint union of
$r\cdot(\abs{\dom B}-1)+1$ copies of~$A$ with an arbitrary linear
ordering~$\preceq^E$; $\iota_E$~is constant. This is an $A$-partite
structure by Proposition~\ref{prop:recti}(\ref{rec-5}).
Since~-- by the same Proposition~-- $B$~is an ordering of the disjoint union of $\abs{\dom B}$ copies
of~$A$, any substructure of~$E$ on $\abs{\dom B}$ elements is isomorphic
to~$B$. In any $r$-colouring of~$\dom E$ there exist $\abs{\dom B}$ elements
of the same colour, inducing a monochromatic copy of~$B$.

If $\abs{\dom A}\geq2$, assume that $\dom A=\{0,1,\dotsc,n\}$. Let $A'$ be the
substructure of~$A$ induced by the subset $\{1,\dotsc,n\}$; let
$B'$ be the substructure of~$B$ induced by $\iota^{-1}_B[\{1,\dotsc,n\}]$,
and $\iota_{B'}=\iota_B\restriction \dom B'$. Then $(B',\iota_{B'})$ is
$A'$-rectified. Apply induction to get $A'$-rectified $(E',\iota_{E'})$
such that $(E',\iota_{E'})\to (B',\iota_{B'})^{(A',\iota_{A'})}_{r^k}$,
where $k=r\cdot\left(\abs{\iota^{-1}_B(0)}-1\right)+1$.  Assuming that $\dom
E'\cap\{1,2,\dotsc,k\}=\emptyset$ let $\dom E=\dom E'\cup\{1,2,\dotsc,k\}$
and define $\iota_E(x)=0$ if $x\in\{1,2,\dotsc,k\}$ and
$\iota_E(x)=\iota_{E'}(x)$ otherwise. Let all $(\sigma\cup\tau)$-relations
of~$E$ be defined by~\eqref{eq:recti}; let $\preceq^E$ be an extension
of~$\preceq^{E'}$ that is preserved by~$\iota_E$. Thus $E'$~is the
substructure of~$E$ on $\iota_E^{-1}[\{1,\dotsc,n\}]$.
See Figure~\ref{fig:l52}. Clearly $(E,\iota_E)$ is $A$-rectified.

\begin{figure}[ht]
\begin{center}
\includegraphics[width=4.5in]{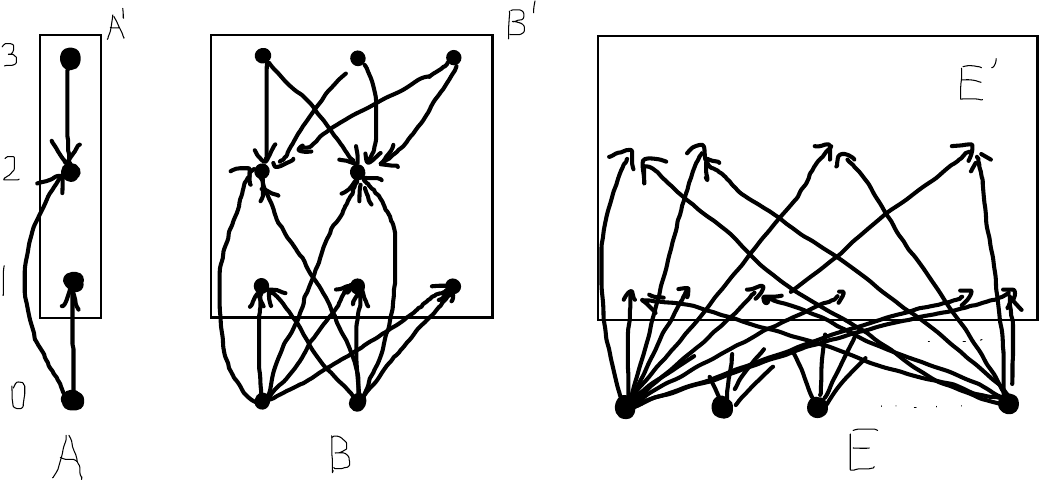}
\caption{Lemma~\ref{lem:partite}.}
\label{fig:l52}
\end{center}
\end{figure}

To prove that $(E,\iota_E)\to(B,\iota_B)^{(A,\id_A)}_r$, consider
any $r$-colouring~$\chi$ of $\tbinom{(E,\iota_E)}{(A,\id_A)}$. Define
$\chi':\tbinom{(E',\iota_{E'})}{(A',\iota_{A'})}\to\{1,\dotsc,r\}^{\iota_E^{-1}(0)}$
by $\chi'(e')=\bigl(c\mapsto\chi(e'\cup(0\mapsto c))\bigr)$.
That is, the $\chi'$-colour of a copy of~$A'$ in~$E'$ is a vector
of $\chi$-colours, one for each of the $k$ extensions of the copy
of~$A'$ by an element in the 0th part to a copy of~$A$ in~$E$.  By the
definition of $(E',\iota_{E'})$, there is a monochromatic
$g'\in\tbinom{(E',\iota_{E'})}{(B',\iota_{B'})}$. Hence
for any fixed $c\in\iota_E^{-1}(0)=\{1,2,\dotsc,k\}$, the mapping
$\phi_c:h'\mapsto \chi((g'\circ h')\cup(0\mapsto c)$) is constant
on $\tbinom{(B',\iota_{B'})}{(A',\id_{A'})}$. Define
$\psi:\iota_E^{-1}(0)\to\{1,\dotsc,r\}$ by setting
$\psi(c)$ to be the constant value of~$\phi_c$. Since
$\abs{\iota_E^{-1}(0)}=k>r\left(\abs{\iota_B^{-1}(0)}-1\right)$, there exists a subset
$M\subseteq\iota_E^{-1}(0)$ with $\abs{M}=\abs{\iota_B^{-1}(0)}$ such that
$\psi$~is constant on~$M$. Define $g\in\tbinom{(E,\iota_E)}{(B,\iota_B)}$
to be an extension of~$g'$ by the $\preceq$-preserving bijection of
$\iota_B^{-1}(0)$ and~$M$.  Then $g$~is monochromatic.
\end{proof}

\section{Partite construction}
\label{sec:constr}

This section is devoted to finishing the proof of Theorem~\ref{thm:main}.
Again, $\F$ is a fixed set of finite $\sigma$-trees and $\vec\C$~is
the ordered expanded class for $\Forbh(\F)$.

\paragraph{Partite structures.}
Let $P$ be an ordered \sigst.  A \emph{$P$-partite $\vec\C$-structure}
is a pair $(A,\iota_A)$ where $A\in\vec\C$ and $\iota_A:\dom A\to\dom
P$ is a homomorphism of the \sigord-reduct~$A^\ast$ of~$A$ to~$P$
that is injective on any tuple of the relation~$R^A$ for any $R\in\sigma$,
and such that the restriction of~$\iota_A$ to any one-element substructure of~$A^\ast$
is an embedding of this one-element \sigord-structure into~$P$.
For an element~$a$ of~$A$ or a tuple~$\bar a$, the image $\iota_A(a)$
or $\iota_A(\bar a)$ is called the \emph{trace} of~$a$ or~$\bar a$.
A $P$-partite $\vec\C$-structure~$(A,\iota_A)$ is \emph{transversal}
if $\iota_A$~is an embedding of~$A^\ast$ to~$P$.

A mapping $e:\dom A\to\dom B$ is an embedding of a $P$-partite
$\vec\C$-structure $(A,\iota_A)$ into $(B,\iota_B)$ if $e:A\eto B$ is an
embedding of \sutordst s and $\iota_A=\iota_B \circ e$.

\begin{lem}[``rectification'']
\label{lem:rectification}
Let $\vec\C$ be the ordered expanded class for\/ $\Forbh(\F)$, where $\F$~is
a set of finite $\sigma$-trees.
Let $(C,\iota_C)$ be a $P$-partite $\vec\C$-structure for some \sigst~$P$.
If $(D,\iota_D)$ is defined by setting
\begin{equation}
\label{eq:rectidef}
\begin{aligned}
&\dom D = \dom C,\\
&\iota_D = \iota_C,\\
&S^D = S^C \text{ for $S\in\tau$},\\
&{\preceq}^D = {\preceq}^C,\\
&\text{for $R\in\sigma$:}\\
&\quad\bar x \in R^D \iff 
\begin{aligned}[t]
&\iota_D\text{ is injective on $\bar x$, and}\\
&\exists\bar y\in R^C\colon \iota_C(\bar y)=\iota_D(\bar x)
\text{ and } \forall i,\ \{S\in\tau\colon x_i\in S^D\} = \{S\in\tau\colon y_i\in S^C\},
\end{aligned}
\end{aligned}
\end{equation}
then $(D,\iota_D)$ is a $P$-partite $\vec\C$-structure.
Moreover, $R^C\subseteq R^D$ for any $R\in\sigma$;
if $R$~is unary, then $R^C=R^D$.
\end{lem}

\begin{proof}
It is straightforward that $\iota_D$ is a homomorphism of the
\sigord-reduct~$D^\ast$ to~$P$ because $\iota_C$~is a homomorphism of~$C^\ast$
to~$P$. By definition, $\iota_D$~is injective on any tuple of any
$\sigma$-relation of~$D$, and every one-element substructure
of~$D$ is isomorphic to the corresponding one-element substructure
of~$C$.

To show that $D\in\vec\C$, first apply the ``only if'' direction
of Lemma~\ref{lem:crit} to prove that the tuple trace of any $\bar
y\in R^C$ is in~$\C$ because $C\in\vec\C$. Then observe that the tuple trace of any
$\bar x\in R^D$ is equal to the tuple trace of some $\bar y\in R^C$.
Also, every one-element substructure of~$D$ is equal to the corresponding
one-element substructure of~$C$. Finally
apply the ``if'' direction of Lemma~\ref{lem:crit}.

If $\bar x\in R^C$ then $\bar y=\bar x$ can be taken to show that
$\bar x\in R^D$. Thus $R^C\subseteq R^D$. If $R$~is unary and $x\in
R^D$, then there is $y\in R^C$ with $\iota_C(y)=\iota_C(x)$; hence
$x\in R^C$ because $\iota_C$ restricted to any one-element structure
is an embedding (by the definition of a $P$-partite structure).
Therefore $R^C=R^D$ for all unary~$R\in\sigma$.
\end{proof}

Observe that the $P$-partite $\vec\C$-structure $(D,\iota_D)$ from
Lemma~\ref{lem:rectification} is \emph{``rectified''} in the following sense:
\begin{multline}
\label{eq:part-rect}
\text{For any $R\in\sigma$ and any $\bar y\in R^D$, if $\bar x$ is a tuple
such that $\iota_D(\bar x)=\iota_D(\bar y)$,}\\
\text{$\iota_D$ is injective on $\bar x$,
and $\{S\in\tau\colon y_i\in S^D\} = \{S\in\tau\colon x_i\in S^D\}$ for all $i$,
then $\bar x\in R^D$.}
\end{multline}
Lemma~\ref{lem:rectification} asserts that any $P$-partite
$\vec\C$-structure $(C,\iota_C)$ can be transformed into $(D,\iota_D)$
that satisfies~\eqref{eq:part-rect} by adding tuples to (non-unary)
$\sigma$-relations.  Note that if $(C,\iota_C)$ already
satisfies~\eqref{eq:part-rect} and $(D,\iota_D)$ is defined
by~\eqref{eq:rectidef}, then no tuples will be added and
$(D,\iota_D)=(C,\iota_C)$. In particular, this is the case if
$(C,\iota_C)$ is transversal.

\paragraph{Rectified substructures.}

The next lemma will apply in the proof of Theorem~\ref{thm:main}
in the following situation: Start with $(D,\iota_D)$ which is
rectified in the above sense, that is, it satisfies~\eqref{eq:part-rect},
and an ordered structure $A\in\vec\C$. Split the elements of~$A$
into parts so that $A$ would be a transversal $P$-partite $\vec\C$-structure
($\iota_A:\dom A\to\dom P$ is an embedding of~$A^\ast$ into~$P$).
Select those elements of~$D$
\begin{enumerate}[label=(\roman*)]
\item whose trace lies in the trace of~$A$, and
\item whose unary $\tau$-relations are exactly the same as the unary
$\tau$-relations of the corresponding element of~$A$.
\end{enumerate}
The selected elements induce a substructure~$B$ of~$D$. There is a
natural way to define $\iota_B:\dom B\to\dom A$ so that an element
of~$B$ would be mapped to the element of~$A$ in the same $P$-part
(that is, $\iota_A(\iota_B(b))=\iota_D(b)$).  Lemma~\ref{lem:rect-rect}
claims that such $(B,\iota_B)$ is $A$-rectified.

\begin{lem}
\label{lem:rect-rect}
Let $(D,\iota_D)$ be a $P$-partite $\vec\C$-structure
satisfying~\eqref{eq:part-rect}, and let $(A,\iota_A)$ be a transversal
$P$-partite $\vec\C$-structure.
Suppose there is a $P$-partite embedding of $(A,\iota_A)$ into $(D,\iota_D)$.
Define
\begin{multline}
\label{eq:domb}
\dom B = \bigl\{x\in\dom D\colon
\iota_D(x)\in\iota_A[\dom A]\\
\text{ and } \{S\in\tau\colon  x\in S^D\} =
\{S\in\tau\colon \iota^{-1}_A(\iota_D(x))\in S^A\} \bigr\}
\end{multline}
and let $B$ be the substructure of~$D$ induced by~$\dom B$. Set
$\iota_B=\iota^{-1}_A\circ(\iota_D\restriction\dom B)$. Then $(B,\iota_B)$ is $A$-rectified.
\end{lem}

\begin{proof}
If $x\preceq^B x'$, then $x\preceq^D x'$ because $B$~is a substructure
of~$D$; thus $\iota_D(x)\preceq^P \iota_D(x')$ because $D$~is
$P$-partite; hence
$\iota_B(x)=\iota_A^{-1}(\iota_D(x))\preceq^A\iota_A^{-1}(\iota_D(x'))=\iota_B(x')$
because $\iota_A$~is a $(\sigma\cup\{\preceq\})$-embedding. If $S\in\tau$ and $x\in\dom B$,
then $x\in S^B$ iff $\iota_B(x)=\iota_A^{-1}(\iota_D(x))\in S^A$
by~\eqref{eq:domb}.  Let $R\in\sigma$ and $\bar x\in R^B\subseteq
R^D$. Then $\iota_D$~is injective on~$\bar x$ because $D$~is
$P$-partite; hence also $\iota_B$~is injective on~$\bar x$. Moreover,
$\iota_B(\bar x)\in R^A$ because $\iota_D$~is a homomorphism and
$\iota_A$~an embedding.

Conversely, suppose that $R\in\sigma$, $\bar x\in(\dom B)^{\ar(R)}$, $\iota_B$~is
injective on~$\bar x$ and $\iota_B(\bar x)\in R^A$. Let $e:A\eto
D$ be an embedding such that $\iota_D\circ e=\iota_A$; let $\bar
y=e(\iota_B(\bar x))$. Then $\bar y\in R^D$ and $\iota_D(\bar
y)=\iota_A(\iota_B(\bar x))=\iota_D(\bar x)$. Moreover, for any~$i$,
$\{S\in\tau\colon y_i\in S^D\} = \{S\in\tau\colon \iota_B(x_i)\in
S^A\}$ because $e$~is an embedding, and $\{S\in\tau\colon \iota_B(x_i)\in
S^A\}=\{S\in\tau\colon x_i\in S^D\}$ by~\eqref{eq:domb}. Therefore
$\bar x\in R^D$ by~\eqref{eq:part-rect}, whence $\bar x\in R^B$
because $B$~is a substructure of~$D$.
\end{proof}

\paragraph{Proof of Theorem~\ref{thm:main}.}
Let $\F$ be a set of finite $\sigma$-trees and let $\C$ be the expanded
class and $\vec\C$ the ordered expanded class for $\Forbh(\F)$. Consider
$A,B\in\vec\C$ and a positive integer~$r$. We construct $C\in\vec\C$
such that $C\to(B)^A_r$.

Let $A^\ast$, $B^\ast$ be the \sigord-reducts of $A$, $B$, respectively. By
Theorem~\ref{thm:nesrod} there exists an ordered \sigst~$P$ such that
$P\to(B^\ast)^{A^\ast}_r$. Define $(C_0,\iota_{C_0})$ by
\begin{align*}
&\dom C_0 = \tbinom{P}{B^\ast}\times\dom B,\\
&\text{for any $k$-ary $R\in\sigma\cup\tau$: }\\
&\qquad\qquad\qquad R^{C_0}=\left\{ ((f,x_1),(f,x_2),\dotsc,(f,x_k))\colon
	f\in\tbinom{P}{B^\ast}\text{ and }
	(x_1,x_2,\dotsc,x_k)\in R^B\right\},\\
&\iota_{C_0}:\dom C_0\to\dom P
	\text{ is defined by }
	\iota_{C_0}:(f,x)\mapsto f(x),\\
&\text{${\preceq}^{C_0}$ is any linear ordering
	that is preserved by $\iota_{C_0}$}.
\end{align*}
Thus $C_0$ (without the ordering) is isomorphic to a sum of structures, and each of the
summands is isomorphic to~$B$.
See Figure~\ref{fig:c0}.
Observe that $(C_0,\iota_{C_0})$ is a $P$-partite $\vec\C$-structure
because $\C$~is closed under taking sums.

\begin{figure}[h]
\begin{center}
\includegraphics[height=2.5in]{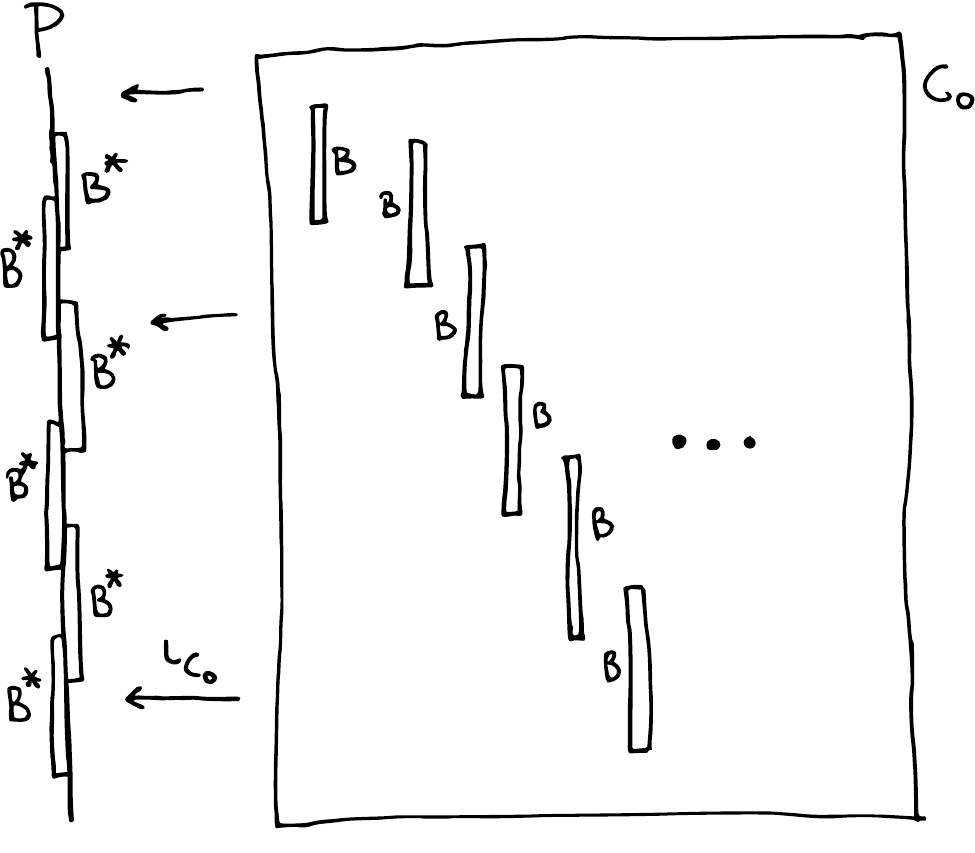}
\end{center}
\caption{$C_0$.}
\label{fig:c0}
\end{figure}

Unless $B$ is connected, $C_0$~may contain other copies of~$B$,
however. Therefore we call the embeddings $c_f:B\eto C_0$ with
$c_f(x)=(f,x)$ for some $f\in\binom{P}{B^\ast}$ \emph{distinguished};
the corresponding substructures of~$C_0$ are called \emph{distinguished
copies} of~$B$. Each of the distinguished copies of~$B$ in~$C_0$
forms a transversal structure. Hence
if $(D_0,\iota_{D_0})$~is obtained from~$(C_0,\iota_{C_0})$
by~\eqref{eq:rectidef}, then each of the distinguished embeddings
of~$B$ to~$C_0$ is also an embedding of~$B$ to~$D_0$.
In other words, for any $R\in\sigma$ none of the new $R$-tuples added by rectification lie
within a distinguished copy of~$B$.

Fix some numbering of $\tbinom{P}{A^\ast}=\{e_1,\dotsc,e_N\}$, the
set of all embeddings of~$A^\ast$ into~$P$. We will inductively
construct $P$-partite $\vec\C$-struc\-tures $(C_1,\iota_{C_1})$,
\dots, $(C_N,\iota_{C_N})$.

Let $k\in\{1,\dotsc,N\}$ and suppose $(C_{k-1},\iota_{C_{k-1}})$
has been constructed. If there is no $P$-partite
embedding of $(A,e_k)$ into $(C_{k-1},\iota_{C_{k-1}})$,%
\footnote{There is a copy $e_k[A^\ast]$ of~$A^\ast$ in~$P$, but it
may not lie within any copy of~$B^\ast$ in~$P$. Even if there is a
copy of~$A^\ast$ in~$C_{k-1}^\ast$ with the same trace, however,
the $\tau$-relations may not be the right ones, so that
these elements induce a substructure of~$C_{k-1}$ that is not
isomorphic to~$A$. In these cases, it can happen that no copy of~$A$
in~$C_{k-1}$ has trace~$e_k[A^\ast]$.}
let $(C_k,\iota_k)=(C_{k-1},\iota_{C_{k-1}})$.  Otherwise let
$(D_{k-1},\iota_{D_{k-1}})$ be defined from $(C_{k-1},\iota_{C_{k-1}})$
by the rectification construction~\eqref{eq:rectidef}. Let
$(B_k,\iota_{B_k})$ be the substructure of $(D_{k-1},\iota_{D_{k-1}})$
obtained as in~\eqref{eq:domb} (Lemma~\ref{lem:rect-rect}), using
$(A,e_k)$ in place of~$(A,\iota_A)$. Then
$(B_k,\iota_{B_k})$ is $A$-rectified and we can apply the Partite Lemma,
Lemma~\ref{lem:partite}, in order to get $A$-rectified $(E_k,\iota_{E_k})$
such that $(E_k,\iota_{E_k})\to (B_k,\iota_{B_k})^{(A,\id_A)}_r$
(w.r.t.~embeddings of $A$-rectified structures). Therefore
$(E_k,e_k\circ\iota_{E_k})\to(B_k,e_k\circ\iota_{B_k})^{(A,e_k)}_r$
(w.r.t.~embeddings of $P$-partite structures).

Now we proceed to construct~$C_k$ from~$E_k$ and several copies
of~$D_{k-1}$ by amalgamation. The construction described below
gives the result explicitly. For each $P$-partite copy of~$B_k$
in~$E_k$, we glue a copy of~$D_{k-1}$ onto~$E_k$, overlapping on
that copy of~$B_k$. Formally, put
\begin{align*}
&\dom C_k = \dom E_k \cup
\left(\tbinom{(E_k,\iota_{E_k})}{(B_k,\iota_{B_k})}\times
(\dom D_{k-1}\setminus\dom{B_k}) \right).\\
\intertext{Define $\lambda_k:\tbinom{(E_k,\iota_{E_k})}{(B_k,\iota_{B_k})}
\times\dom D_{k-1} \to \dom C_k$ by}
&\lambda_k : (g,x) \mapsto \begin{cases}
g(x)& \text{if $x\in\dom B_k$},\\
(g,x)& \text{otherwise},
\end{cases}\\
\intertext{(so $\lambda_k(g,x)$ gives the name of the element
of~$C_k$ corresponding to the element~$x$ in the ``$g$th'' copy
of~$D_{k-1}$ within~$C_k$). For any $\ell$-ary $R\in\sigma\cup\tau$, let}
&R^{C_k} = \Bigl\{ \bigl(\lambda_k(g,x_1),\dotsc,\lambda_k(g,x_\ell)\bigr)\colon
g\in\tbinom{(E_k,\iota_{E_k})}{(B_k,\iota_{B_k})},\ 
(x_1,\dotsc,x_\ell)\in R^{D_{k-1}} \Bigr\}.\\
\intertext{Furthermore define $\iota_{C_k}:\dom C_k\to\dom P$ by}
&\iota_{C_k}: y \mapsto e_k(\iota_{E_k}(y)) \quad\text{if $y\in\dom E_k$},\\
&\iota_{C_k}: (g,x) \mapsto \iota_{D_{k-1}}(x) \quad\text{otherwise.}
\end{align*}
Note that $\iota_{C_k}(\lambda_k(g,x))=\iota_{D_{k-1}}(x)$ for any
$x\in\dom D_{k-1}$ and $g\in\binom{(E_k,\iota_{E_k})}{(B_k,\iota{B_k})}$.
Finally, let $\preceq^{C_k}$ be a linear ordering
such that $y\preceq^{C_k}y'$ if $y\preceq^{E_k}y'$,
$\lambda_k(g,x)\preceq^{C_k}\lambda_k(g,x')$ if $x\preceq^{D_{k-1}}x'$,
and $z\preceq^{C_k}z'$ if $\iota_{C_k}(z)\preceq^P\iota_{C_k}(z')$.
See Figure~\ref{fig:ck}.

\begin{figure}[h]
\begin{center}
\includegraphics[height=2.5in]{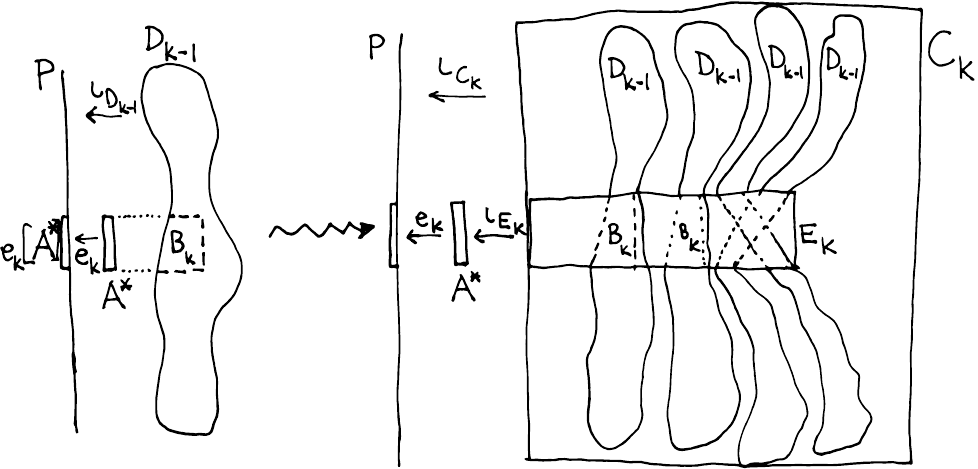}
\end{center}
\caption{$C_k$.}
\label{fig:ck}
\end{figure}

Notice that for a fixed~$g$, the mapping $\lambda_k(g,{-}):x\mapsto
\lambda_k(g,x)$ is an embedding of $(D_{k-1},\iota_{D_{k-1}})$ to
$(C_k,\iota_{C_k})$. By definition of~$D_{k-1}$, $\lambda_k(g,{-})$~is
an injective homomorphism of $(C_{k-1},\iota_{C_{k-1}})$ to
$(C_k,\iota_{C_k})$. The inclusion mapping is an embedding of~$E_k$
to~$C_k$ because $(E_k,\iota_{E_k})$ is $A$-rectified.

Now we claim that $(C_k,\iota_{C_k})$ is a $P$-partite $\vec\C$-structure.
Every one-element substructure of~$C_k$ is isomorphic to a one-element
substructure of~$D_{k-1}$, and every tuple of some relation~$R^{C_k}$,
$R\in\sigma$, corresponds to some tuple of~$R^{D_{k-1}}$ with the same tuple trace. Since
$D_{k-1}\in\vec\C$, by Lemma~\ref{lem:crit} we have $C_k\in\vec\C$.
To show that $\iota_{C_k}$~is a homomorphism, let $\bar y\in R^{C_k}$
for some $R\in\sigma$. Then $\bar y=\lambda_k(g,\bar x)$ for some
$g\in\binom{(E_k,\iota_{E_k})}{(B_k,\iota_{B_k})}$ and $\bar x\in
R^{D_{k-1}}$, and $\iota_{C_k}(\bar y)=\iota_{D_{k-1}}(\bar x)$.
Thus $\iota_{C_k}(\bar y)\in R^P$ because $\iota_{D_{k-1}}$~is a
homomorphism. By definition, $\iota_{C_k}$~also preserves~$\preceq^{C_k}$.
Hence $\iota_{C_k}:C_k^\ast\hto P$ is a homomorphism of \sigord-structures.

Next, to show that $\iota_k$~is injective on any $\bar y\in R^{C_k}$
(that is, $\bar y$~does not contain two distinct elements in the
same part), again use the fact that $\bar y =\lambda_k(g,\bar x)$
for some $g\in\binom{(E_k,\iota_{E_k})}{(B_k,\iota_{B_k})}$ and
$\bar x\in R^{D_{k-1}}$ and that $\lambda_k(g,-)$ preserves the
parts. Finally, let $y\in\dom C_k$. Then $y =\lambda_k(g,x)$ for
some $g\in\binom{(E_k,\iota_{E_k})}{(B_k,\iota_{B_k})}$ and $x\in
\dom{D_{k-1}}$. The one-element substructure of~$C_k^\ast$ induced
by~$y$ is isomorphic to the one-element substructure of~$D_{k-1}^\ast$
induced by~$x$, which in turn is isomorphic to the one-element
substructure of~$P$ induced by~$\iota_{D_{k-1}}(x)$ because
$D_{k-1}$~is $P$-partite. As $\iota_{C_k}(y)=\iota_{D_{k-1}}(x)$,
we have that the one-element substructure of~$C_k^\ast$ induced
by~$y$ is isomorphic to the one-element substructure of~$P$ induced
by~$\iota_{C_k}(y)$. Therefore $(C_k,\iota_{C_k})$ is a $P$-partite
$\vec\C$-structure.

Let $C=C_N$. We show that $C\to(B)^A_r$. Consider any colouring
$\chi:\tbinom{C}{A}\to\{1,\dotsc,r\}$. By downward induction we find
injective homomorphisms $h_k:(C_{k-1},\iota_{C_{k-1}})\hto(C_k,\iota_{C_k})$
for $k=N,N-1,\dotsc,1$ that have certain monochromatic properties.
Each $h_k=\lambda(g_k,-)$ for some $g_k\in\binom{(E_k,\iota_{E_k})}{(B_k,\iota_{B_k})}$.

Suppose $h_i$ is known for $i=N,\dotsc,k+1$ (possibly for no $i$ yet). If
$(C_k,\iota_{C_k})=(C_{k-1},\iota_{C_{k-1}})$, let $h_k$
be the identity mapping.%
\footnote{This is the case if there was no $P$-partite embedding
of $(A,e_k)$ into $(C_{k-1},\iota_{C_{k-1}})$; see previous footnote.}
Otherwise define the colouring
$\chi_k:\tbinom{(E_k,\iota_{E_k})}{(A,\id_A)}\to\{1,\dotsc,r\}$ by
setting $\chi_k(q)=\chi(h_N\circ h_{N-1}\circ \dotsb\circ  h_{k+1}\circ
q)$. (Observe that the composed mapping is indeed an embedding.) Since
$(E_k,\iota_{E_k})\to(B_k,\iota_{B_k})^{(A,\id_A)}_r$, there exists a
$\chi_k$-monochromatic embedding $g_k:(B_k,\iota_{B_k})\eto(E_k,\iota_{E_k})$. Let
$h_k=\lambda(g_k,{-})$.

Let $h=h_N\circ h_{N-1}\circ \dotsb \circ h_1: (C_0,\iota_{C_0})
\hto (C_N,\iota_{C_N})$.  By definition, $h$~is in fact the following
composed mapping:
\[ C_0 \xrightarrow{\id} D_0 \xrightarrow{h_1}
   C_1 \xrightarrow{\id} D_1 \xrightarrow{h_2}
   C_2 \xrightarrow{\id} D_2 \xrightarrow{h_3}
   \dotsb
    \xrightarrow{\id} D_{N-1} \xrightarrow{h_N} C_N , \]
where each of the identity mappings is a bijective homomorphism
obtained implicitly from Lemma~\ref{lem:rectification} and each
$h_k=\lambda_k(g_k,-)$ is an embedding. In general, $h$~is not an
embedding because new tuples of $\sigma$-relations are added during
rectification. We want to show, however, that no new tuples are
added to the distinguished copies of~$B$ in~$C_0$.
In other words, for any distinguished embedding $c_f: B\eto C_0$,
the mapping $h\circ c_f$ is an embedding: By definition, $h\circ
c_f$~is injective. For $R\in\sigma$, if $\bar x\in R^B$, then
$h(c_f(\bar x))\in R^C$ because $h$~is a homomorphism and $c_f$~an
embedding. If $h(c_f(\bar x))\in R^C$, then $f(\bar x)=\iota_{C_0}(c_f(\bar
x))=\iota_{C_N}(h(c_f(\bar x)))\in R^P$ because $\iota_{C_N}$~is a
homomorphism; hence $\bar x\in R^B$ because $f$~is an embedding.
For $S\in\tau$, we have $x\in S^B$ iff $h(c_f(x))\in S^C$ because
of~\eqref{eq:rectidef} and because $c_f$ and each~$h_k$~is an
embedding.

Consider
any $e_j\in\tbinom{P}{A^\ast}$. Any embedding~$d_j$ of~$A$ to~$C_0$
such that $\iota_{C_0}\circ d_j=e_j$ is also a $P$-partite embedding of
$(A,e_j)$ to $(C_0,\iota_{C_0})$. Moreover, $h\circ d_j$~is a $P$-partite
embedding of $(A,e_j)$ to $(C_N,\iota_{C_N})$. By definition of~$h_j$,
all such embeddings take the same colour under~$\chi$. Thus we define
$\chi_0:\tbinom{P}{A^\ast}\to\{1,\dotsc,r\}$ by $\chi_0(e_j)=\chi(h\circ d_j)$
if there exists $d_j\in\tbinom{C_0}{A}$ such that $\iota_{C_0}\circ d_j=e_j$,
and arbitrarily otherwise. By definition of~$P$ there exists
$\chi_0$-monochromatic $f\in\tbinom{P}{B^\ast}$. Let $c_f:B\eto C_0$ be
the distinguished embedding given by $c_f:x\mapsto (f,x)$.

Conclude the proof by observing that $h\circ c_f$~is a $\chi$-monochromatic
embedding of~$B$ to~$C$.
\qed

\section{A note on Datalog and Constraint Satisfaction}
\label{sec:datalog}

This section contains a brief description of constraint satisfaction
problems (CSPs), Datalog programs and their connection to Ramsey
theory, which was the original motivation for this research.
A more thorough introduction to CSPs and their complexity is given, e.g., in~\cite{CohJea:CCL,CSPSurveys}.
More details on Datalog can be found in~\cite{Kol:On-the-expressive-power};
a concise exposition, relevant to our setting, is given in~\cite{BulKroLar:Dualities}.

\paragraph{Tree Datalog.}

Let $\sigma$, $\tau$ be disjoint finite relational signatures such
that $\tau$~contains only unary relation symbols and a special
nullary relation symbol \goal.
A \emph{tree Datalog program} is a finite set of rules of the form
\[ S(x) \from t_1,\dotsc,t_n \]
or
\[ \goal \from t_1,\dotsc,t_n, \]
where $S\in\tau$ and each $t_i$ is an atomic formula $R_i(x_{i_1},\dotsc,x_{i_k})$
with $R_i\in \sigma\cup\tau$, so that at most one of the $R_i$'s belongs to~$\sigma$.
The part of the rule to the left of the arrow is called the \emph{head} of the rule;
the part to the right is called the \emph{body} of the rule.
In the context of a Datalog program, the predicates appearing in
the head of a rule are called \emph{IDBs} (\emph{intensional database predicates}),
whereas the predicates from~$\sigma$ are called \emph{EDBs} (\emph{extensional database predicates}).

A Datalog program can be ``executed'' on a \sigst~$A^\ast$ to recursively
construct the $\tau$-relations~$S^A$, $S\in\tau$, and consequently
a $(\sigma\cup\tau)$-expans\-ion~$A$ of~$A^\ast$, by repeatedly
adding elements of~$A^\ast$ to the unary relations~$S^A$ following
the program's rules.
The execution terminates when the application of any rule does not
result in adding an element into an IDB.
The \goal\ predicate is initially set to \false, and we say that
the Datalog program \emph{accepts}~$A^\ast$ if its \goal\ predicate
evaluates to \true\ on~$A^\ast$.

A Datalog program can be used to provide a finite description of
an infinite regular set of forbidden trees, as the following example
illustrates.

\begin{exa}
We revisit the example of $\F$ consisting of thunderbolts from
Section~\ref{sec:expand} (Figure~\ref{fig:thunder}).
The context is digraphs, so $\sigma$~contains one binary relation symbol~$A$.
The signature~$\tau$ obtained from Definition~\ref{dfn:expanded}
contains four relation symbols: $S_i$~for each~$\M_i$, $i=1,2,3,4$,
which will be IDBs of the corresponding Datalog program.
In addition to these unary IDBs, the Datalog will have a nullary IDB \goal.
The rules of the program are:
\begin{align*}
S_1(a) &\from  A(b,a); \\
S_1(a) &\from  A(a,b),\ S_2(b); \\
S_2(a) &\from  A(b,a),\ S_1(b); \\
S_3(a) &\from  A(a,b); \\
S_3(a) &\from  A(b,a),\ S_4(b); \\
S_4(a) &\from  A(a,b),\ S_3(b); \\
\Goal  &\from  S_1(a),\ S_4(a); \\
\Goal  &\from  S_2(a),\ S_3(a).
\end{align*}
This Datalog program accepts any given \sigst~$A^\ast$ if and only
if $A^\ast$~admits a homomorphism from some element of~$\F$.
If the program rejects~$A^\ast$, then it constructs a
$(\sigma\cup\tau)$-expansion~$A$ of~$A^\ast$, which will be the
canonical structure given by~\eqref{eq:canon}.
\end{exa}

Datalog provides an explanation for the membership tests for the
expanded class given in Section~\ref{sec:expand}.
Lemma~\ref{lem:one} corresponds precisely to the situation where
the \goal\ predicate is set to \true\ by the Datalog program
(meaning that $A^\ast$~is not $\F$-free).
In Lemma~\ref{lem:crit}, the tuple trace condition corresponds to
reaching a fixed point of the Datalog program, that is, the application
of no rule results in adding an element into an IDB.

\paragraph{Constraint satisfaction problems.}

For a \sigst~$H$, let $\Csp(H)=\{A\colon \exists f: A\hto H \}$.
Given a fixed \sigst~$H$, the \emph{non-uniform constraint satisfaction
problem} is to decide, for an input \sigst~$A$, whether $A\in\Csp(H)$ or not.
The problem's computational complexity depends on~$H$;
many polynomial-time cases can be explained by the existence of a
``nice'' \emph{obstruction set}~$\F$ such that $\Csp(H)=\Forbh(\F)$.
Following~\cite{HelNesZhu:DualPoly}, we say that $H$~has \emph{tree duality}
if $\Csp(H)=\Forbh(\F)$ for
\begin{equation}
\label{eq:treedu}
\F = \{ F\colon F \text{ is a $\sigma$-tree and there is no } f:F\hto H\}.
\end{equation}

For $\F$ given by \eqref{eq:treedu}, the $\pieq$-equivalence class
of any piece $(M,m)$ of some $F\in\F$ is fully determined by the set
\[ \Hh(M,m) = \{ f(m)\colon f:M\hto H \}, \]
because $\I(M,m)=\{(N,n)\colon (N,n) \text{ is a rooted $\sigma$-tree s.t.\ } \Hh(M,m) \cap \Hh(N,n) =\emptyset \}$.
Thus the expanded signature will always be finite and we can attempt
to index the $\tau$-relations by subsets of the domain of~$H$ (in
correspondence with the sets $\Hh(M,m)$).

What we get will be the canonical tree Datalog program:
The EDBs of the program are~-- as always~-- the relations in~$\sigma$.
There is an IDB~$S_X$ for every proper subset $X\subset\dom H$ 
(in the end we will only use the subsets which are definable in~$H$
by a positive existential first-order formula);
$S_X$~is unary unless $X=\emptyset$:
$S_\emptyset$~is nullary and we identify it with the \goal\ predicate.
Moreover, to simplify the description of the program's rules,
identify $S_{\dom H}$ with \true.
Given any $R\in\sigma$ of arity~$r$, $j\in\{1,2,\dotsc,r\}$ and nonempty
sets $X_i\subseteq\dom H$ for all $i\in\{1,2,\dotsc,r\}\setminus\{j\}$, put
\[ X_j = \{x_j\in\dom H\colon \exists (x_1,x_2,\dotsc,x_r)\in R^H \text{ s.t.\ } x_i\in X_i \text{ for each } i\ne j\}. \]
If $X_j\ne\dom H$, introduce the rule
\[ S_{X_j}(a_j) \from R(a_1,a_2,\dotsc,a_r),\ 
	S_{X_1}(a_1),\ \dotsc,\ S_{X_{j-1}}(a_{j-1}),\ S_{X_{j+1}}(a_{j+1}),\ \dotsc,\ S_{X_r}(a_r). \]
(At this point, subsets $X$ not definable in~$H$ by a positive
existential first-order formula will not appear in the head of any
rule and we can drop them~-- as well as any rules containing them
in their body.)

\begin{exa}
In our example (thunderbolts), it is well known that $\Forbh(\F)=\Csp(H)$
for $H=P_2$, the directed path $0\to1\to2$.
The description above results in the following Datalog program:
\begin{align*}
S_{\{0,1\}}(a) &\from  A(a,b); \\
S_{\{1,2\}}(a) &\from  A(b,a); \\
S_{\{0\}}(a) &\from  A(a,b),\ S_{\{0,1\}}(b); \\
S_{\{1,2\}}(a) &\from  A(b,a),\ S_{\{0,1\}}(b); \\
S_{\{0,1\}}(a) &\from  A(a,b),\ S_{\{1,2\}}(b); \\
S_{\{2\}}(a) &\from  A(b,a),\ S_{\{1,2\}}(b); \\
\Goal        &\from  A(a,b),\ S_{\{0\}}(b); \\
S_{\{1\}}(a) &\from  A(b,a),\ S_{\{0\}}(b); \\
S_{\{0\}}(a) &\from  A(a,b),\ S_{\{1\}}(b); \\
S_{\{2\}}(a) &\from  A(b,a),\ S_{\{1\}}(b); \\
S_{\{1\}}(a) &\from  A(a,b),\ S_{\{2\}}(b); \\
\Goal        &\from  A(b,a),\ S_{\{2\}}(b).
\end{align*}
You may notice that the program is different to the one we derived
from the thunderbolts. This is because the obstruction set~$\F$ has
changed: it now contains not only the thunderbolts, but also all
other trees that do not admit a homomorphism to~$P_2$.
\end{exa}

\paragraph{Pigeonhole classes.}

For \sutst s $B$, $C$ and a positive integer~$r$, let $C\to(B)_r^1$
denote the following statement:
Whenever the elements of~$C$ are coloured with $r$ colours,
there exists an embedding $g:B\eto C$ such that for any $b_1,b_2\in\dom B$,
if $g(b_1)$ and $g(b_2)$ induce isomorphic one-element substructures of~$C$,
then $g(b_1)$ and $g(b_2)$ have the same colour.
We say that a class~$\C$ of \sutst s is a \emph{pigeonhole class}
if for any structure $B\in\C$ and positive integer~$r$ there exists
$C\in\C$ such that $C\to(B)_r^1$.

Consider a set~$\F$ of $\sigma$-trees and let $\C$ be the expanded
class for $\Forbh(\F)$. Suppose that $\C$~contains finitely many
non-isomorphic one-element structures: this is certainly the case
if the expanded signature $\sigma\cup\tau$ is finite, in particular,
if $\F$~is given by~\eqref{eq:treedu}.
Then it follows by repeated application of Theorem~\ref{thm:main}
that the expanded class is a pigeonhole class (in fact, it follows
for the ordered expanded class, but if we only colour one-element
substructures, the ordering is not needed).
It has already been proved by Atserias and
Weyer~\cite{AtsWey:Decidable-Relationships} that any class with
free amalgamation is a pigeonhole class.

\paragraph{Characterising tree duality.}

Following the approach of~\cite{AtsWey:Decidable-Relationships}
further, let $H$ be a \sigst, let $\F$ be given by~\eqref{eq:treedu}
and let $\C$ be the expanded class for $\Forbh(\F)$.
We need to construct $V\in\C$ such that any tuple trace appearing
in~$\C$ will be an induced substructure of~$V$: we can take the sum
(disjoint union) of all such tuple traces.
Now let $W\in\C$ satisfy $W\to(V)_r^1$ for $r=\left|\dom H\right|$.
Define $\sim$ on $\dom V$ by putting $v_1\sim v_2$ if and only if
$v_1$ and $v_2$ induce isomorphic one-element substructures of~$V$
and put $U=V/{\sim}$.
Finally, let $U^\ast$, $V^\ast$ and~$W^\ast$ be the $\sigma$-reducts
of $U$, $V$ and~$W$, respectively.
Using the pigeonhole property of~$W$, one gets the following:

\begin{prop}[\cite{AtsWey:Decidable-Relationships}]
The following conditions are equivalent:
\begin{enumerate}[label=\({\alph*}]
\item There exists a homomorphism $W^\ast\hto H$.
\item There exists a homomorphism $U^\ast\hto H$.
\item $\Csp(H)=\Forbh(\F)$.
\qed
\end{enumerate}
\end{prop}

Thus we have found for any~$H$ a \sigst~$U^\ast=U^\ast(H)$ such
that $H$~has tree duality if and only if there is a homomorphism
$U^\ast(H)\hto H$.
This $U^\ast(H)$ appears to be intimately related to the power
structure of~\cite{FedVar:SNP}.

Lastly, we only mention
in passing that another connection between constraint satisfaction
problems and Ramsey classes is studied in~\cite{BodPin:Reducts-of-Ramsey}.

\section{Final comments}
\label{sec:final}

\paragraph{Universal structures.}

If $\F$ is a set of finite $\sigma$-trees, then by
Fra\"\i ss\'e's Theorem~\cite{Fraisse53}, Theorem~\ref{thm:amalgamation}
implies that the expanded class~$\C$ for $\Forbh(\F)$ has a Fraïssé
limit: a countable homogeneous \sutst~$U$ such that $\C$~is the class of all
finite substructures of~$U$. The $\sigma$-reduct~$U^\ast$ of~$U$
is a universal structure for $\Forbh(\F)$. For finite~$\F$ this
universal structure~$U^\ast$ is \ocat; the existence of such a
universal \ocat\ structure (and much more) was proved by Cherlin,
Shelah and Shi~\cite{CheSheShi:Universal}. If $\F$~is infinite,
$U^\ast$~is no longer necessarily \ocat\
(see~\cite{HubNes:Universal-structures} and the next paragraph);
however, it is model-complete.

\paragraph{Regular classes of trees.}

Recall that two pieces $(M,m)$, $(M',m')$ are $\pieq$-equivalent
if their incompatible sets are equal, that is, if $\I(M,m)=\I(M',m')$.
By Definition~\ref{dfn:expanded}, the signature~$\tau$ is finite
if and only if $\pieq$~has finitely many equivalence classes on the
pieces of the trees contained in~$\F$. In this case, we call~$\F$
a \emph{regular class} of $\sigma$-trees; the term is motivated by
a connection to regular languages, highlighted
in~\cite{ErdTarTar:Caterpillar-dualities}.  This definition of
regularity coincides with the one from~\cite{HubNes:Universal-structures}.
In~\cite{ErdPalTar:On-infinite-finite-tree-duality}, however, a
set~$\F$ of trees is defined to be regular if $\pieq$~has finitely
many equivalence classes on all rooted $\sigma$-forests. Let us
call such a set~$\F$ \emph{EPTT-regular}.  Obviously, every
EPTT-regular set is regular, but the converse does not hold.

Let $\UP(\F) = \{F\colon F \text{ is a $\sigma$-tree and there
exists } F'\in\F \text{ s.t. } F'\hto F \}$.
Obviously, $\Forbh(\UP(\F))=\Forbh(\F)$.
For a (not necessarily
finite) \sigst~$H$, define $\Csp(H) = \{A\colon A\hto H\}$.
From~\cite{ErdPalTar:On-infinite-finite-tree-duality,HubNes:Universal-structures}
and the results of this paper, we can conclude:

\begin{thm}
Let $\sigma$ be a finite relational signature and let $\F$ be a set
of finite $\sigma$-trees. Then the following are equivalent:
\begin{enumerate}[label=\({\alph*}]

\item $\UP(\F)$ is regular;

\item $\UP(\F)$ is EPTT-regular;

\item $\Forbh(\F)=\Csp(H)$ for some finite $\sigma$-structure~$H$;

\item there is a countable \ocat\ \sigst~$U^\ast$ universal for $\Forbh(\F)$;

\item there is a countable homogeneous \sutst~$U$ over a finite
signature $\sigma\cup\tau$ such that the $\sigma$-reduct of~$U$ is
universal for~$\Forbh(\F)$;

\item there is a Ramsey class~$\vec\C$ of ordered \sutst s over a
finite signature $\sigma\cup\tau$ such that $\Forbh(\F)$ is the
class of the $\sigma$-reducts of the structures in~$\vec\C$.
\qed
\end{enumerate}
\end{thm}

\paragraph{Extreme amenability.}
By a theorem of Kechris, Pestov and Todorčević~\cite{Topo-Dynamics},
the automorphism group of a Ramsey structure is extremely amenable. Thus
Theorem~\ref{thm:main} provides a continuum of examples of structures
with an extremely amenable automorphism group: take $\F'$ to be an infinite
antichain of $\sigma$-trees; then the Fraïssé limit of the ordered expanded
class for $\Forbh(\F)$ provides such an example for any subset~$\F$
of~$\F'$.

\paragraph{Problem.}
It would be interesting to classify all sets~$\F$ of \sigst s
for which the corresponding ordered expanded class for $\Forbh(\F)$ is
a Ramsey class.  In particular, is it the case for any set~$\F$ of
connected finite \sigst s?

\paragraph{Limits of the partite method.}
Nešetřil~\cite{Nes:CSS-Ramsey} asked whether one can prove all Ramsey
classes by a variant of the partite (amalgamation) construction. This
is certainly a question worth considering. It is not very satisfactory
that the definition of a partite structure is rather different each time:
compare~\cite{BotFon:Ramsey,Nesetril04,NesRod:Simple,NesRod:Two-proofs,NesRod:Ramsey,NesRod:Strong,NesRod:The-partite,NesetrilHandbookChapter}.
Also, the partite lemma is sometimes proved by induction
(as in~\cite{BotFon:Ramsey,ProVoi:A-short} and here),
sometimes by an application of the Hales--Jewett theorem (as
in~\cite{NesRod:Strong,NesRod:The-partite,NesetrilHandbookChapter}).

\section*{Acknowledgements}

I thank Albert Atserias, Manuel Bodirsky, Jarik Nešetřil and Claude Tardif for many
stimulating discussions, intriguing questions and useful comments.
Many thanks go to the anonymous referees whose questions and
objections have led to this paper being much better than it would
otherwise have been.

I am fully supported by a fellowship from the Swiss National Science
Foundation. Part of this work was done at Laboratoire d'Informatique
(LIX), CNRS UMR 7161, Ecole Polytechnique, Paris, where I was
supported by ERC Starting Grant ``CSP-Complexity'' no.~257039 of
Manuel Bodirsky. Another part was done when I was visiting Fields
Institute in Toronto within the Summer Thematic Program on the
Mathematics of Constraint Satisfaction.  Finally, a part was done
at Coffee\&Company in Kingston, Ontario. I also acknowledge the
support of the Institute for Theoretical Computer Science (ITI),
project 1M0545 of the Ministry of Education of the Czech Republic,
and EPSRC grant EP/I01795X/1.

\end{document}